\documentclass[a4paper,11pt,twoside]{amsart}
\usepackage{amssymb}
\usepackage{enumitem}
\usepackage[colorlinks = true,
            linkcolor = blue,
            urlcolor  = blue,
            citecolor = blue,
            anchorcolor = blue]{hyperref}
\usepackage{caption}
\newtheorem{theorem}{Theorem}[section]

\newtheorem{proposition}{Proposition}[section]
\newtheorem{lemma}[theorem]{Lemma}
\newtheorem{corollary}[theorem]{Corollary}
\theoremstyle{definition}

\usepackage{colortbl}

\usepackage{geometry}
\theoremstyle{remark}
\newtheorem{remark}[theorem]{Remark}
\DeclareMathOperator{\sech}{sech}
\DeclareMathOperator{\Var}{Var}
\DeclareMathOperator{\Cov}{Cov}
\DeclareMathOperator{\Li}{Li}

\numberwithin{equation}{section}
\usepackage{graphicx}
\graphicspath{ {./images/} }
\def\MR#1{\href{http://www.ams.org/mathscinet-getitem?mr=#1}{MR#1}}

\begin{document}

\title{Real roots of random polynomials: asymptotics of the variance}      
   \author{Yen Q. Do}
\address{Department of Mathematics, University of Virginia,  Charlottesville, Virginia 22904, USA}
\email{yqd3p@virginia.edu}

\author{Nhan D. V. Nguyen}

\address{Department of Mathematics, University of Colorado Boulder, Boulder, Colorado 80309, USA \and Department of Mathematics and Statistics, Quy Nhon University, Binh Dinh, Vietnam}
\email{Nhan.Nguyen-1@colorado.edu, nguyenduvinhan@qnu.edu.vn}

\date{\today}
\begin{abstract} We compute the precise leading asymptotics of the variance of the number of real roots for a large class of random polynomials, where the random coefficients have polynomial growth. Our results apply to many classical ensembles, including the Kac polynomials, hyperbolic polynomials, their derivatives, and any linear combinations of these polynomials. Prior to this paper, such asymptotics were established only for the Kac polynomials in the 1970s, with the seminal contribution of Maslova. The main ingredients of the proof are new asymptotic estimates for the two-point correlation function of the real roots, revealing geometric structures in the distribution of the real roots of these random polynomials. As a corollary, we obtain asymptotic normality for the real roots of these random polynomials, extending and strengthening a related result of O. Nguyen and V. Vu.
\end{abstract} 
\maketitle
\section{Introduction} 
In this paper, we consider the real roots of random polynomials of the form
\begin{equation} \label{e.ranpoly}
    P_n(x)=\sum_{j=0}^n \xi_j c_jx^j,
\end{equation}
where $\xi_0,\dots,\xi_n$ are independent real-valued random variables with unit variance, and $c_0,\dots, c_n$ are deterministic real-valued coefficients.  

Let $N_n$ denote the number of real roots of $P_n$. The distribution of $N_n$ is a classical topic in probability theory with a long and rich history. While leading asymptotics for the expectation of $N_n$ in the large $n$ limit have been established for many classical ensembles of random polynomials, analogous results for the variance of $N_n$ are harder to come by, especially for random polynomials with non-Gaussian coefficients, and this is the main goal of the current paper.

We begin with a brief history of the subject before discussing the class of random polynomials studied in this work. As noted by Todhunter \cite[p. 618]{Tod14}, Waring was already exploring the distribution of ``impossible'' (non-real) roots of typical polynomial equations as early as 1782. However, it took nearly 150 years for concrete estimates to appear. The first such result was given by Bloch and P\'olya in 1932 \cite{BP}, who considered random polynomials with independent coefficients taking values in $\{-1, 0, 1\}$ with equal probability and established a crude upper bound of $O(\sqrt{n})$ for the average number of real roots. Here and throughout, we use standard asymptotic notation as $n \to \infty$. For real sequences $\{\alpha_n\}$ and $\{\beta_n\}$, we write $\alpha_n = O(\beta_n)$, $\alpha_n \ll \beta_n$, or $\beta_n \gg \alpha_n$ to mean $|\alpha_n| \le C\beta_n$ for some constant $C>0$ independent of $n$. Moreover, $\alpha_n = o(\beta_n)$ indicates that $\alpha_n / \beta_n \to 0$ as $n\to \infty$. 

The result of Bloch and P\'olya \cite{BP} has inspired numerous subsequent studies on the (average) number of real roots of random polynomials with independent and identically distributed (i.i.d.) coefficients. This became the {\em primary theme} of much early work, with seminal contributions by Littlewood--Offord \cite{LO1, LO2, LO3}, Kac \cite{K, K1}, and Ibragimov--Maslova \cite{IM1, IM2, IM3}. In particular, for random polynomials with i.i.d. centered Gaussian coefficients, Kac \cite{K} derived an exact formula for $\mathbb E[N_n]$ and showed that
\begin{equation} \label{eq1.2}
    \mathbb E[N_n]  =\frac{2}{\pi}\log n+o(\log n).
\end{equation}
A key component of Kac’s method is a formula for the mean density of real roots, also known as the one-point correlation function (see, e.g., \cite{TV}; see also \cite{EK} for a geometric interpretation). The Kac formula and its generalization, the Kac--Rice formula, are now among the most fundamental tools for studying the real roots of Gaussian random polynomials. Moreover, random polynomials with i.i.d. coefficients are commonly referred to as {\em Kac polynomials} \cite{TV}.

The asymptotic formula \eqref{eq1.2} has been extended to Kac polynomials with non-Gaussian coefficients; see \cite{EO, IM1, IM3, K1, NNV, St}. Notably, Nguyen--Nguyen--Vu \cite{NNV} improved the error term in \eqref{eq1.2} to $O(1)$ under broad assumptions on the coefficient distribution. Furthermore, for several interesting classes of coefficient distributions, it was shown in \cite{DHV} that $\mathbb E [N_n]=\frac 2{\pi}\log n+C+o(1)$, where the constant $C$ may depend on the coefficient distribution. For Gaussian polynomials, Wilkins \cite{W} derived a complete asymptotic expansion of $\mathbb E[N_n]$.

In the 1970s, Maslova \cite{M, M1} initiated {\em two new themes} to the subject: the leading asymptotics for the variance $\Var[N_n]$, 
\begin{equation} \label{e.kacvar}
    \Var[N_n] =\left[\frac{4}{\pi}\left(1-\frac{2}{\pi}\right)+o(1)\right]\log n,
\end{equation}
and the convergence (in distribution) of the standardization of $N_n$ to a standard Gaussian,  
\begin{equation}\label{e.kacclt}
    \frac{N_n -\mathbb E[N_n]}{\sqrt{\Var[N_n]}} \xrightarrow{d} \mathcal N(0,1).
\end{equation}
The convergence in \eqref{e.kacclt} is often referred to as a central limit theorem (CLT), or the asymptotic normality of the real roots of $P_n$. Maslova's work has inspired numerous authors to explore similar results for other classes of random polynomials and random functions, in both Gaussian and non-Gaussian settings; see, e.g., \cite{AL21, AADL, AADL1, ADL16, AL13, BCP19, BS, BD97, BD04, D15, DNN22, DV20,  F, Gas23,  GW11,  LP21, LP22, NV22A, NV22D, STB83, SM}.  The current work is part of these themes.

While earlier studies of the real roots of random polynomials with i.i.d. coefficients were largely motivated by mathematical curiosity, it turns out that there is a deeper connection to the theory of random point processes in statistical physics. In the beautiful exposition \cite{HKPV09}, Hough et al. examined random point processes that are invariant under the group of isometries of the underlying Riemann surfaces. In two dimensions, this classification leads to three classical ensembles of random point processes, corresponding to the isometry groups of the complex plane $\mathbb C$, the Riemann sphere $\mathbb S^1 = \mathbb C \cup \{\infty\}$, and the hyperbolic disk $\mathbb D=\{|z|<1\}$:
\begin{enumerate}
    \item random flat point processes, given by the complex zeros of the random flat (or Weyl) series,  $\sum_{j=0}^\infty\frac 1 {\sqrt{j!}} \xi_j z^j$, where $\{\xi_j\}_{j\ge 0}$ are independent standard complex Gaussian random variables;
    \item random elliptic point processes, given by the complex zeros of the random elliptic polynomials, $\sum_{j = 0}^n \sqrt{\binom{n}{j}} \xi_j z^j$, where $\{\xi_j\}_{j\ge 0}$ are independent standard complex Gaussian random variables; and
    \item random hyperbolic point processes, given by the complex zeros of the random hyperbolic series, $\sum_{j=0}^\infty \sqrt{\frac{L(L+1)\dots (L+j-1)}{j!}} \xi_j z^j$, where $\{\xi_j\}_{j\ge 0}$ are independent standard complex Gaussian random variables, and $L>0$ is a fixed parameter.
\end{enumerate}

A central theme in the field is to consider the real analogues of these random analytic functions, where the $\xi_j$ are real-valued random variables, and to study the distribution of the real roots of their polynomial approximations (except in the elliptic case, where the model is already polynomial). In particular, Kac polynomials arise as polynomial approximations of the random hyperbolic series with parameter $L=1$. Furthermore, higher-dimensional generalizations of random elliptic polynomials have been studied; see \cite{AADL, AADL1, K93, SS93}.

We now recall the relevant existing results on variance asymptotics and asymptotic normality for the real roots of the above ensembles of random polynomials. In what follows, the Gaussian setting refers to the case where all random coefficients $\xi_j$ are standard Gaussian. For random flat polynomials, variance asymptotics and CLTs for the number of real roots were established in the Gaussian setting in a joint work of V. Vu and the first author \cite{DV20}. For Gaussian elliptic polynomials, the leading asymptotics for variances were computed by Bleher--Di \cite{BD97} (see also \cite{N}), and the CLT was proved by Dalmao \cite{D15} (see also \cite{AL21} and \cite{N}). These results were later extended to higher dimensions in \cite{AADL, AADL1}. To our knowledge, non-Gaussian extensions of \cite{AL21, AADL, AADL1, BD97, D15, DV20} remain open problems. As mentioned above, Maslova \cite{M, M1} established both variance asymptotics and CLTs for Kac polynomials under broad conditions on the coefficients. For random hyperbolic polynomials and their derivatives, CLTs were proved for a wide class of coefficient distributions in Nguyen--Vu \cite{NV22D}. However, the methods in \cite{NV22D} do not yield precise leading asymptotics for the variance, which is addressed in the present paper. We also highlight related results for random functions: variance asymptotics have been studied in \cite{BCP19, DNN22, Gas23, GW11,  LP21, LP22}, and CLTs in \cite{AL21, ADL16, AL13, DNNP21}.

\subsection{Statement of main results}
As discussed above, among the three classical ensembles of random polynomials, random hyperbolic polynomials remain the ensemble for which the precise leading asymptotics for the variance of the number of real roots are not known (even for the Gaussian setting), other than the special case of Kac polynomials, treated by Maslova in the 1970s. The current paper addresses this open problem. Furthermore, our methods work for a very general class of random polynomials, which includes not only hyperbolic polynomials but also their derivatives and any linear combination of these polynomials, and our results also apply to both Gaussian and non-Gaussian settings.\footnote{Aside from Maslova’s result from five decades ago, the only known non-Gaussian results for variance asymptotics pertain to random trigonometric functions; see \cite{BCP19, DNN22}.}

We first state our result for random hyperbolic polynomials. To this end, we fix some notation. Given $\tau>-1/2$, define
\begin{equation} \label{ftau}
    f_{\tau}(u):=\left(\sqrt{1-\Delta_{\tau}^2(u)}+\Delta_{\tau}(u) \arcsin \Delta_{\tau}(u)\right)\Sigma_{\tau}(u)-1,
\end{equation}
where 
\begin{align} \label{del}  
    \Delta_{\tau}(u)&:=u^ {\tau+1/2}\frac{u(1-u^{2\tau+1})-(2\tau+1)(1-u)}{1-u^{2\tau+1}-(2\tau+1)u^{2\tau+1}(1-u)},\\
    \notag    \Sigma_{\tau}(u)&:=\frac{1-u^{2\tau+1}-(2\tau+1)(1-u)u^{2\tau+1}}{(1-u^{2\tau+1})^{3/2}},
\end{align}
and let
\begin{equation} \label{kappa}
    \kappa_\tau:=\left(\frac{2\tau+1}{\pi}\int_{0}^{\infty}f_{\tau}(\sech^2v)dv+\frac{\sqrt{2\tau+1}}{2}\right)\frac{1}{\pi}.
\end{equation}
Some basic properties of $f_\tau$, including its integrability, are collected in Lemma \ref{l.f-fact}.

\begin{theorem}[Variance asymptotics for real roots of random hyperbolic polynomials]\label{t.hyperbolic-main} Let $\xi_0,\dots,\xi_n$ be real-valued independent random variables with zero mean, unit variance, and uniformly bounded $(2+\varepsilon)$-moments for some $\varepsilon>0$. Fix $L>0$ and consider the random hyperbolic polynomial
    \[
    P_{n,L}(x)=\xi_0 +\sqrt L \xi_1 x+ \dots  +\sqrt{\frac{L(L+1)\cdots (L+n-1)}{n!}} \xi_n x^n.
    \]
For any $k\ge 0$, let $N_{n,k}$ be the number of real roots of the $k$th derivative of $P_{n,L}$ (so $k=0$ means $P_{n,L}$ itself). Then for $\tau=k+\frac {L-1}2$, we have
    \[
    \Var[N_{n,k}]=\left[2\kappa_\tau+\frac 2\pi \Big(1-\frac 2\pi\Big)+o(1)\right]\log n.
    \]
\end{theorem}
\begin{remark} When $L=1$, the random hyperbolic polynomial $P_{n, L}$ becomes a random Kac polynomial, so Theorem \ref{t.hyperbolic-main}  applies to the derivatives of Kac polynomials, which are novel even in the Gaussian setting.  Note that for $L=1$ and $k=0$, we have $\tau=0$ and
\begin{align*}
    \int_0^\infty f_\tau(\sech^2v)dv &=\int_0^\infty \left(\tanh^2v+\tanh v\sech v \arcsin(\sech v)-1\right)dv=\frac{\pi}{2}-2,
\end{align*}
thus, it follows from \eqref{kappa} that 
    \[
    \kappa_\tau=\frac{1}{\pi}\left(1-\frac{2}{\pi}\right),
    \]
and this recovers Maslova's result given in \eqref{e.kacvar} for the Kac polynomials. 
\end{remark}

Theorem \ref{t.hyperbolic-main} is a special case of the following more general result for random polynomials with coefficients having polynomial asymptotics. Below, we collect the assumptions for this slightly more technical theorem.  We assume that there are fixed positive constants  $C_0$, $C_1$, $C_2$, $N_0$, $\varepsilon$, and a fixed constant $\tau>-1/2$,  such that
\begin{enumerate}
\renewcommand{\labelenumi}{\theenumi}
\renewcommand{\theenumi}{(A\arabic{enumi})}
    \item \label{A1} $\xi_0,\dots,\xi_n$ are independent real-valued random variables,  with $\mathbb E[\xi_j]=0$ for $j\ge N_0$, $\Var[\xi_j]=1$ for $j\ge 0$, and $\sup_{0\le j\le n}\mathbb E[|\xi_j|^{2+\varepsilon}]<C_0$, and
    \item \label{A2} $c_0,\dots, c_n$ are real coefficients such that for $N_0\le j\le n$,  we have\footnote{The $o_j(1)$ notation in \eqref{e.polyasymp} means that this term can be bounded by some $o_j$ independent of $n$ such that $\lim_{j\to\infty} o_j=0$.}
\begin{equation}\label{e.polyasymp}
    |c_j|=C_1 j^\tau(1+o_j(1)),  
\end{equation}
and for $0\le j <N_0$, we have 
    \[
    |c_j|\le C_2.
    \]
\end{enumerate}
We emphasize that the coefficients $c_j$ may vary with $n$; otherwise, the second condition in \ref{A2} would be superfluous.

\begin{theorem}[Variance asymptotics for real roots of generalized Kac polynomials] \label{t.main} Assume that the polynomial $P_n$ defined by \eqref{e.ranpoly} satisfies conditions \ref{A1} and \ref{A2}. Then,
    \[
    \Var[N_n]=\left[2\kappa_\tau +\frac{2}{\pi}\left( 1-\frac{2}{\pi}\right)+o(1)\right]\log n,
    \]
where the implicit constants in the $o(1)$ term depend only on $N_0$, $C_0$, $C_1$, $C_2$, $\varepsilon$, $\tau$, and the decay rate of $o_j(1)$ in condition \ref{A2}.  Furthermore,  $N_n$ satisfies the CLT; that is, the following convergence in distribution to a normalized Gaussian holds as $n\to\infty$:
    \[
    \frac{N_n-\mathbb E[N_n]}{\sqrt{\Var[N_n]}} \xrightarrow{d} \mathcal N(0,1).
    \]
\end{theorem}

Note that the leading asymptotics 
$\left[2\kappa_\tau +\frac{2}{\pi}\left( 1-\frac{2}{\pi}\right)+o(1)\right]\log n$
naturally decomposes into two terms. The first, $[2\kappa_\tau+o(1)]\log n$, reflects the contribution from real roots in the interval $(-1,1)$, while the second, $\left[\frac{2}{\pi}\left( 1-\frac{2}{\pi}\right)+o(1)\right]\log n$, arises from real roots outside this interval. We will also establish that the real roots inside and outside $(-1,1)$ are asymptotically independent. Consequently, the growth parameter $\tau$ affects only the distribution of real roots within $(-1,1)$, and $\log n \ll \Var[N_n]$. It then follows that the second part of Theorem~\ref{t.main}, concerning the asymptotic normality of the number of real roots of $P_n$, is a consequence of the conditional CLT of O. Nguyen and V. Vu \cite[Theorem 1.2]{NV22D} for real roots of generalized Kac polynomials.

\begin{proposition}[\cite{NV22D}] 
Let $P_n$ be defined by \eqref{e.ranpoly}, where $\xi_0,\dots,\xi_n$ satisfy condition \ref{A1}. Assume that for some fixed positive constants $C_1$, $C_2, $ $C_3$, and $N_0$, and a constant $\tau>-1/2$, we have
\begin{equation} \label{e.polygrowth}
\begin{cases}
    C_1j^{\tau}\le |c_j|\le C_2j^{\tau}&\text{if}\quad N_0\le j\le n,\quad \text{and}\\
    c_j^2\le C_3&\text{if}\quad 0\le j<N_0.
\end{cases}
\end{equation}
Further, assume
\begin{equation}\label{e.lowbound}
    \log n \ll \Var[N_n].
\end{equation}
Then $N_n$ satisfies the CLT.
\end{proposition}

The elegant proof in \cite{NV22D}  employs universality to reduce to the Gaussian case (i.e., i.i.d. standard Gaussian $\xi_j$) with roots restricted to a core region, and then adapts Maslova's method \cite{M1}, which approximates the number of roots by a sum of independent random variables.

In \cite{NV22D}, the lower bound \eqref{e.lowbound} was established for a class of random polynomials satisfying the following additional condition
\begin{equation} \label{e.OV}
    \frac{|c_j|}{|c_n|}-1=O\left(e^{-(\log\log n)^{1+\varepsilon}}\right),\quad  n-ne^{-\log^{1/5}n}\le j \le n-e^{\log^{1/5}n}.
\end{equation}
This condition was later verified for Kac polynomials, their derivatives, and hyperbolic polynomials, leading to the CLTs for the real roots of these polynomials in \cite{NV22D}.  Our asymptotic estimates for the variance of the real roots in these cases strengthen these results, as they provide details about the denominator of $\frac{N_n-\mathbb E[N_n]}{\sqrt{\Var[N_n]}}$.

We also note that the class of random polynomials in Theorem~\ref{t.main} and the class of generalized Kac polynomials satisfying Nguyen--Vu's condition \eqref{e.OV} are substantially different. For the convenience of the reader, we include an example in Appendix \ref{a.example} that satisfies \eqref{e.polyasymp} but does not satisfy \eqref{e.OV}.

\begin{remark} The assumption $\tau > -1/2$ plays a critical role throughout our proof. It would be interesting to explore generalized Kac polynomials in the regime $\tau \leq -1/2$. In particular, it remains an open question whether the methods developed in this paper can be extended to estimate $\Var[N_n]$ in that setting, a direction we leave for future investigation. For a recent comprehensive overview of these polynomials, we refer the reader to the work of Krishnapur, Lundberg, and Nguyen \cite{KLN}, which provides asymptotic results for the expected number of real roots and offers insights into bifurcating limit cycles.
\end{remark}

\begin{remark} Under the same assumption as in  Theorem~\ref{t.main},  it was shown in \cite{DNV} (see also \cite{D}) that 
\begin{equation} \label{eq1.5}
    \mathbb E[N_n] = \frac{1+\sqrt{2\tau+1}}{\pi}\log n +o(\log n).
\end{equation}
For generalized Kac polynomials satisfying \eqref{e.polygrowth} (but not necessarily \eqref{e.polyasymp}), it was also shown in \cite{DNV} that $\mathbb E [N_n]$ grows logarithmically in $n$. The asymptotic formula \eqref{eq1.5} in special cases where the $\xi_j$ are Gaussian and $c_j=j^{\tau}$ for some $\tau>0$ was previously considered in \cite{D1, D2, S,SGF,SM0, SM}.  We note that \eqref{eq1.5} also recovers Maslova's result in \cite[Theorem 2]{M} for the first derivatives of Kac polynomials.
\end{remark}

\begin{remark} Generalized Kac polynomials have also attracted attention in the mathematical physics community. In particular, Schehr and Majumdar \cite{SM0, SM} established a connection between the persistence exponent of the diffusion equation with random initial conditions and the probability that a certain generalized Kac polynomial has no real root in a given interval. A more comprehensive treatment was later provided by Dembo and Mukherjee \cite{DM}, who derived general criteria for the continuity of persistence exponents for centered Gaussian processes. These criteria were then applied to study gap probabilities for both the real roots of random polynomials and the zero-crossings of solutions to the heat equation with Gaussian white noise initial data.
\end{remark}

\begin{remark} As a numerical illustration, consider the polynomial $P_n(x)=\sum_{j=1}^n \xi_j j x^j$, which can essentially be viewed as the derivative of a Kac polynomial. By Theorem~\ref{t.main}, we have
    \[
    \Var[N_n]= C\log n +o(\log n),
    \]
where
    \[
    C=2\kappa_1+ \frac{2}{\pi}\left(1-\frac{2}{\pi}\right) \approx 0.575737845.
    \]
Figure~\ref{fig1} and Figure~\ref{fig2} present numerical simulations supporting this theoretical result. Moreover, the numerical evidence in Figure~\ref{fig3} appears to support the conjecture that $\Var[N_n] - C\log n$ converges to a finite limit as $n\to \infty$, with the limiting value potentially depending on the distribution of $\xi_j$.

\begin{figure}[htbp]
\centering
\includegraphics[scale=1.0]{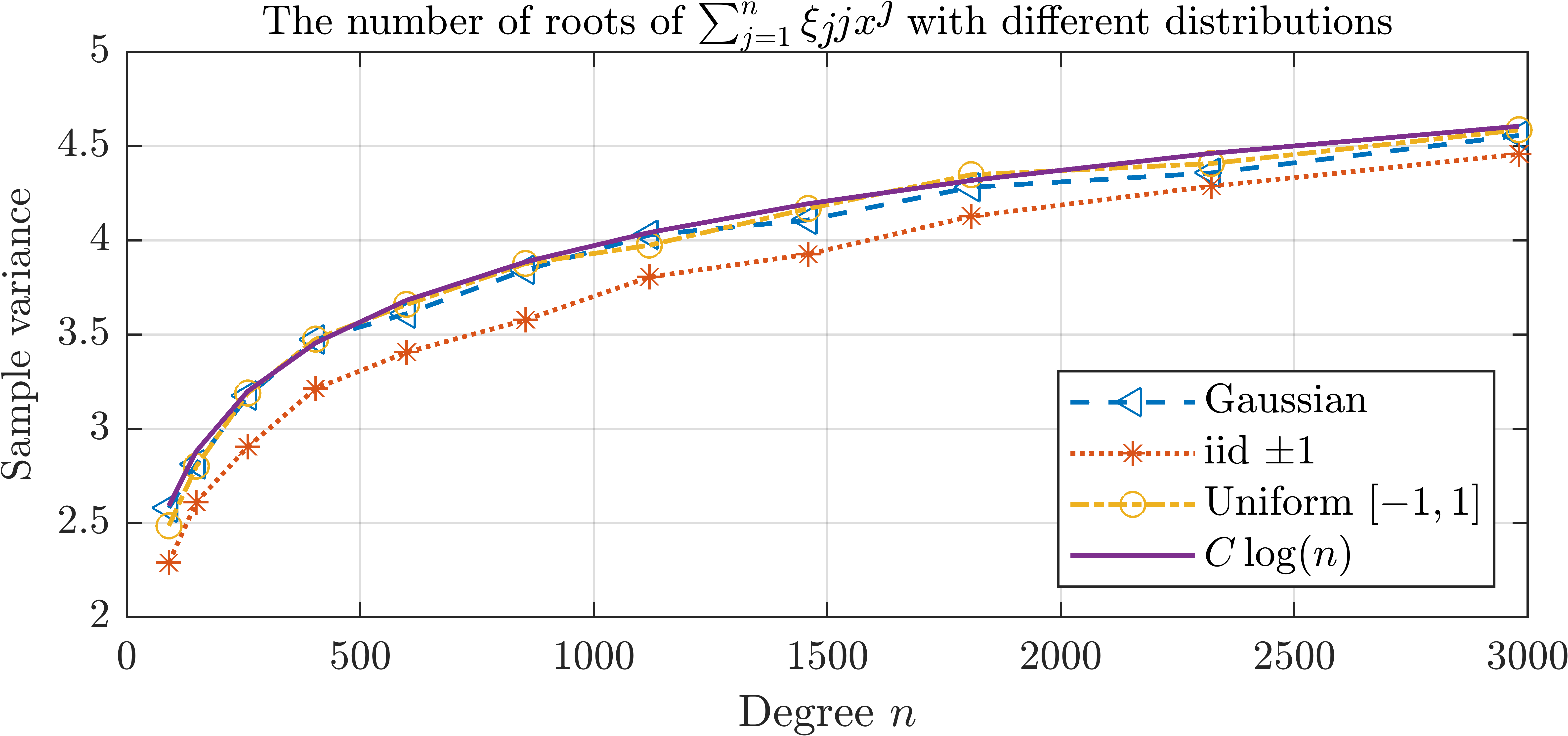}
\caption{Plot of sample variances versus the degree $n$.}
\label{fig1}
\end{figure}

\begin{figure}[htbp]
\centering
\includegraphics[scale=1.0]{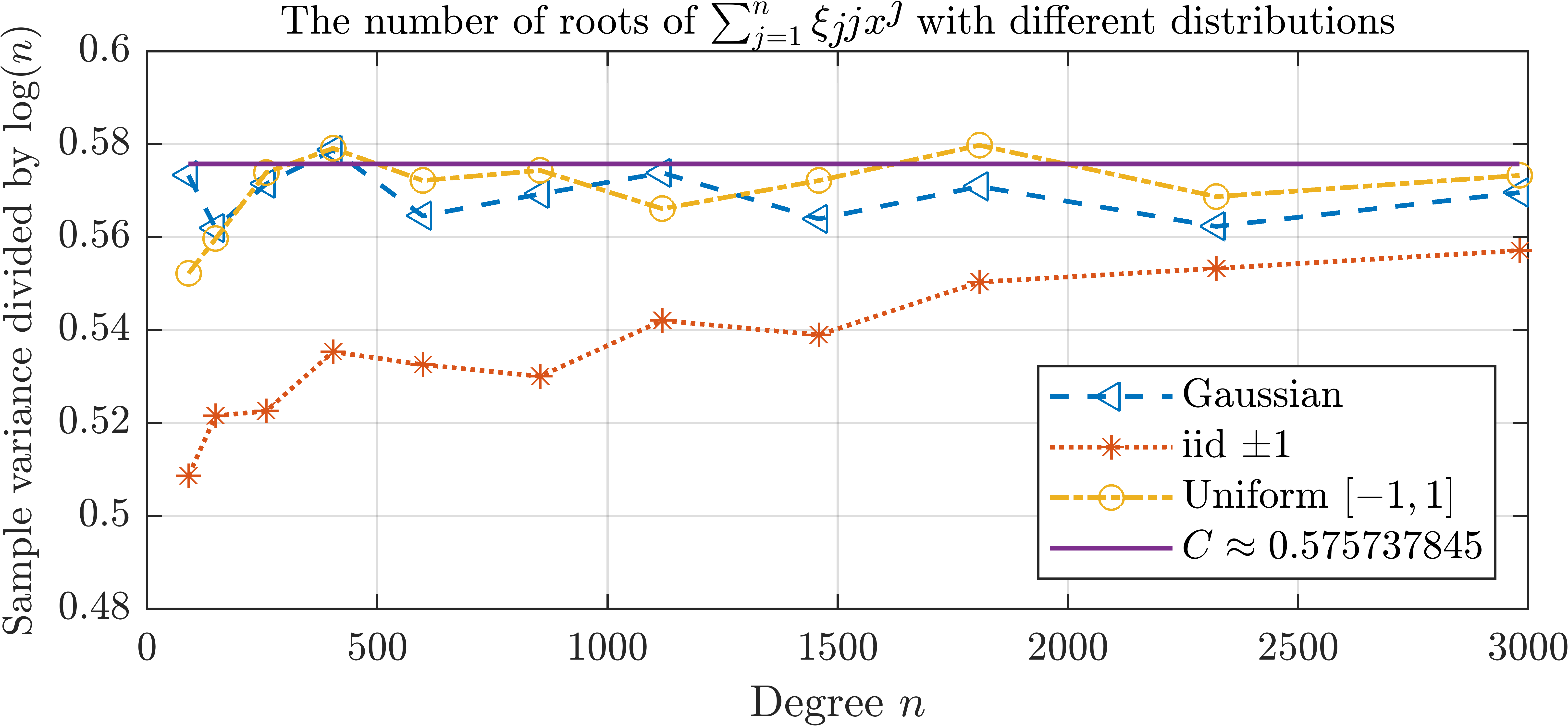}
\caption{Sample variances divided by $\log(n)$ are  approaching $C$.}
\label{fig2}
\end{figure}

\begin{figure}[htbp]
\centering
\includegraphics[scale=1.0]{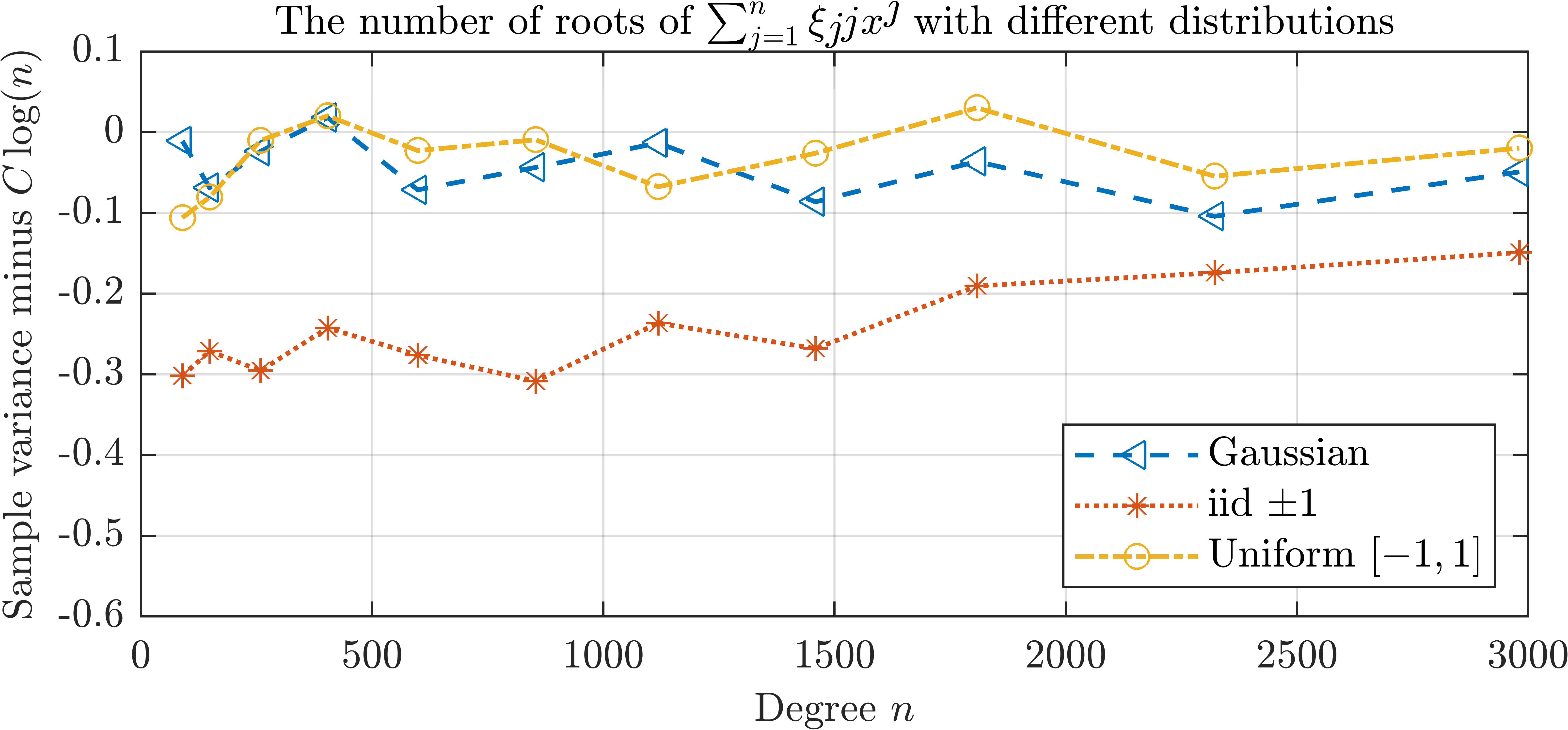}
\caption{We subtract $C\log(n)$ from the sample variances, and the curves seem to converge to different values.}
\label{fig3}
\end{figure}
\end{remark}

\subsection{Main ideas of the proof}
We now discuss the main ideas behind our proof. To prove Theorem \ref{t.main}, our starting point is the universality argument of Nguyen and Vu in \cite{NV22D}, which reduces the problem to the Gaussian case.  However, our consideration of the Gaussian setting differs from the Nguyen--Vu argument. In \cite{NV22D}, the authors used a novel swapping argument to compare the Gaussian version of $P_n$ with a classical Kac polynomial (based on assumption \eqref{e.OV} and the reciprocal formulation of $P_n$) and deduced the lower bound \eqref{e.lowbound} from Maslova's variance estimate for the number of real roots in $[-1,1]$. This elegant approach, however, only accounts for the real roots of $P_n$ outside $[-1,1]$ and the consequential lower bound in \cite{NV22D} for the variance is unfortunately not sharp. 

In contrast, our approach avoids treating Maslova’s estimates as black boxes. The main ingredients in our evaluation of the variance for the Gaussian case are new asymptotic estimates for the two-point correlation function of the real roots of $\widetilde P_n$ (the Gaussian analogue of $P_n$), as developed in Section~\ref{sec3}. These estimates reveal some underlying hyperbolic geometry inside the distribution of the real roots of generalized Kac polynomials (especially under hypothesis \eqref{e.polyasymp}): the asymptotics depend on a certain notion of pseudo-hyperbolic distance between the real roots. One of the main challenges in the proof is the fact that there are various instances when one has to find the leading asymptotics for an algebraic sum where the asymptotics of the summands may negate each other, particularly when the mentioned pseudo-hyperbolic distance is very small. A separate consideration is often required in such situations, where a geometric property of this distance (Lemma~\ref{l.ez}) becomes useful in the proof. 

We would also like to mention that Maslova's proof in \cite{M} for the Kac polynomials relies on very delicate computations for the variance of real roots inside carefully chosen local intervals. Extending such explicit computations to the setting of the current paper, where there are no closed-form expressions for the coefficients $c_j$, appears to be quite challenging. Nevertheless, the estimates for the correlation functions in the current paper can certainly be utilized to derive asymptotic estimates for the variances of the number of real roots inside local intervals, and we include some local estimates in Theorem~\ref{t.maingauss}.

The remainder of the paper is organized as follows. In Section~\ref{sec2}, we recall the universality method from \cite{NV22D}, which allows a reduction to the Gaussian case. Section~\ref{sec3} presents the estimates for correlation functions, and the proof of the Gaussian case is given in Section~\ref{s.mainthmproof}.

\subsection{Notational conventions}
For any subset $I$ of $\mathbb R$, where $I$ may depend on $n$, we denote by $N_n(I)$ the number of real roots of $P_n(x)$ in $I$ (counted with multiplicity). In particular, $N_n = N_n(\mathbb R)$ is the total number of real roots. These random variables may take values in $\{0,1,...,n\}$.

\section{Reduction  to the Gaussian case} \label{sec2}
We begin by recalling the universality method developed in \cite{NV22D} for generalized Kac polynomials. The core idea of this method is to show that the limiting laws of real roots are largely independent of the distribution of the coefficients $\xi_j$. This is achieved via a replacement principle, which compares correlation functions of two random functions whose log-magnitudes are similarly distributed and satisfy certain anti-concentration conditions. Once universality is established, it suffices to prove the desired law in the Gaussian case, where powerful tools such as the Kac--Rice formula and properties of Gaussian processes are available. This method is particularly effective for analyzing local statistics, such as root densities and correlation functions in regions where the expected number of real roots is $O(1)$. However, since the framework is inherently local, extending it to global laws requires several technically involved steps to link local behavior to global statistics. To this end, the authors of \cite{NV22D} partition the real line into two regions: a core region $\mathcal{I}_n$, which contains the majority of real roots, and the complement $\mathbb{R} \setminus \mathcal{I}_n$, which contains only a negligible number.

To adapt this approach to our setting, we define the core region $\mathcal{I}_n$ for $P_n$ as follows. Fix $d \in (0, 1/2)$, and let
    \[
    d_n:=\exp(\log^{\frac d4}n),\quad a_n:=d_n^{-1},\quad b_n:=n^{-1}d_n,\quad \text{and}\quad I_n:=[1-a_n,1-b_n].
    \]
We then let
    \[
    \mathcal I_n=I_n\cup (-I_n) \cup I_n^{-1} \cup (- I_n^{-1}),
    \]
where for any given set $S$, we define $-S:=\{-x: x\in S\}$ and $S^{-1}:=\{x^{-1}: x\in S\}$.

Let $\widetilde{P}_n(x)$ denote the Gaussian analogue of $P_n(x)$, defined by
    \[
    \widetilde{P}_n (x):=\sum_{j=0}^n\widetilde{\xi}_j c_jx^j,
    \]
where $\widetilde{\xi}_0, \dots, \widetilde{\xi}_n$ are real-valued, independent standard Gaussian random variables, and the coefficients $c_0, \dots, c_n$ satisfy assumption \ref{A2}. For $S\subset \mathbb R$, we denote by $\widetilde{N}_n (S)$ the number of real roots of $\widetilde{P}_n(x)$ inside $S$. 

As shown in \cite[Theorem 2.1 and Corollary 2.2]{NV22D}, the distribution of the roots of $P_n$ closely approximates that of its Gaussian analogue $\widetilde{P}_n$ within the core region $\mathcal I_n$.
\begin{proposition}[\cite{NV22D}] \label{pro.u}
Let $P_n$ be defined as in \eqref{e.ranpoly}, where $\xi_0,\dots,\xi_n$ satisfy condition \ref{A1}, and $c_0,\dots, c_n$ satisfy \eqref{e.polygrowth}. There exist positive constants $c$ and $\lambda$ such that for sufficiently large $n$ and every function $F:\mathbb R\to \mathbb R$ with derivatives up to order three bounded by 1, one has
    \[
    |\mathbb E[F(N_n(\mathcal I_n))]-\mathbb E[F(\widetilde{N}_n(\mathcal I_n))]|\le ca_n^{\lambda}+cn^{-\lambda}.
    \]
In particular, for any fixed integer $k\ge 1$, 
    \[
    |\mathbb E[N_n^k(\mathcal I_n)]-\mathbb E[\widetilde{N}_n^k(\mathcal I_n)]|\le ca_n^{\lambda}+cn^{-\lambda}, 
    \]
and
    \[
    |\Var[N_n(\mathcal I_n)]-\Var[\widetilde{N}_n(\mathcal I_n)]|\le ca_n^{\lambda}+cn^{-\lambda}.
    \]
\end{proposition}

Moreover, as a consequence of \cite[Proposition 2.3]{NV22D}, the contribution from outside the core region $\mathbb R\setminus \mathcal I_n$ is negligible.

\begin{proposition}[\cite{NV22D}] \label{pro.out}
Let $P_n$ be defined as in \eqref{e.ranpoly}, where $\xi_0,\dots,\xi_n$ satisfy condition \ref{A1}, and $c_0,\dots, c_n$ satisfy \eqref{e.polygrowth}. Then, for any integer $k\ge 2$, there exists a constant $c>0$ such that
    \[
    \mathbb E[N_n^k(\mathbb R\setminus \mathcal I_n)] \le c(\log a_n)^{2k}.
    \]
\end{proposition}

With the aid of Propositions \ref{pro.u} and \ref{pro.out}, Theorem \ref{t.main} will be proved once we prove the following theorem for the Gaussian case.

\begin{theorem}[Gaussian case] \label{t.maingauss} Fix $S_n\in \{-I_n, I_n\}$. As $n\to \infty$, it holds that
\begin{align*}
    \Var[\widetilde{N}_n(S_n)] &=\left(\kappa_\tau+o(1)\right)\log n,\\
    \Var[\widetilde{N}_n(S_n^{-1})] &=\bigg[\frac{1}{\pi}\left(1-\frac{2}{\pi}\right)+o(1)\bigg]\log n,
\end{align*}
and
    \[
    \Var[\widetilde{N}_n(\mathcal I_n)]=\bigg[2\kappa_\tau +\frac{2}{\pi}\left( 1-\frac{2}{\pi}\right)+o(1)\bigg]\log n,
    \]
where $\kappa_\tau$ is given by \eqref{kappa} and the implicit constants in the $o(1)$ terms depend only on the constants $N_0$, $C_1$, $C_2$, $\tau$, and the rate of decay of $o_j(1)$ in condition \ref{A2}.
\end{theorem}

\section{Estimates for the correlation functions} 
\label{sec3}
The proof of Theorem \ref{t.maingauss} relies on the Kac--Rice formulas for the expectation and variance, which will be recalled below.

\subsection{Kac--Rice formulas} Let $r_n(x,y)$ denote the normalized correlator of $\widetilde{P}_n$, defined as
    \[
    r_n(x,y):=\frac{\mathbb E[\widetilde{P}_n (x)\widetilde{P}_n (y)]}{\sqrt{\Var[\widetilde{P}_n (x)]\Var[\widetilde{P}_n (y)]}}.
    \]
Setting $k_n(x):=\sum_{j=0}^nc_j^2x^j$, we see that 
\begin{equation} \label{r}
    r_n(x,y)=\frac{k_n(xy)}{\sqrt{k_n(x^2)k_n(y^2)}}.
\end{equation}
Thanks to the Kac--Rice formulas, the one-point correlation function of the real roots of $\widetilde P_n$ is given by (see, e.g., \cite{EK})
    \[
    \rho^{(1)}_n(x)=\frac{1}{\pi}\sqrt{\frac{\partial^2r_n}{\partial x\partial y}(x,x)},  
    \]
and the two-point correlation function is given by (see, e.g., \cite[Lemma 2.1]{N})
    \[
    \rho^{(2)}_n(x,y)=\frac{1}{\pi^2}\left(\sqrt{1-\delta^2_n(x,y)}+\delta_n(x,y) \arcsin \delta_n(x,y)\right)\frac{\sigma_n(x,y)}{\sqrt{1-r_n^2(x,y)}},
    \]
where 
    \[
    \sigma_n(x,y)=\sqrt{\bigg(\frac{\partial^2r_n}{\partial x\partial y}(x,x)-\frac{(\frac{\partial r_n}{\partial x}(x,y))^2}{1-r_n^2(x,y)}\bigg)\bigg(\frac{\partial^2r_n}{\partial x\partial y}(y,y)-\frac{(\frac{\partial r_n}{\partial y}(x,y))^2}{1-r_n^2(x,y)}\bigg)}
    \]
and 
    \[
    \delta_n(x,y)=\frac{1}{\sigma_n(x,y)}\bigg(\frac{\partial^2r_n}{\partial x\partial y}(x,y)+\frac{r_n(x,y)\frac{\partial r_n}{\partial x}(x,y)\frac{\partial r_n}{\partial y}(x,y)}{1-r_n^2(x,y)}\bigg).
    \]
As is standard, for any interval $I\subset \mathbb R$, we have 
\begin{equation} \label{mean}
    \mathbb E[\widetilde{N}_n (I)]=\int_{I}\rho^{(1)}_n(x)dx
\end{equation}
and 
\begin{equation} \label{var}
    \Var[\widetilde{N}_n (I)]=\iint_{I\times I}\left[\rho^{(2)}_n(x,y)-\rho^{(1)}_n(x)\rho^{(1)}_n(y)\right]dydx +\mathbb E[\widetilde{N}_n (I)].
\end{equation}

\subsection{Asymptotic estimates for the variance function and its derivatives} The estimates for the variances rely on the formulas \eqref{mean} and \eqref{var}, and asymptotic estimates for the correlation functions $\rho_n^{(1)}$ and $\rho_n^{(2)}$, which will be established shortly. To this end, we first investigate the behavior of $r_n(x,y)$ for $x,y\in I_n \cup (-I_n)$, and thanks to \eqref{r}, this will be done via estimates for $k_n(x)$ for $|x|\in I_n^2:=\{uv: u, v\in I_n\}$. Recall that $I_n=[1-a_n,1-b_n]$, where $a_n=d_n^{-1}=\exp(-\log^{d/4}n)$ and $b_n=d_n/n$. In what follows, we will assume that $n$ is sufficiently large and $S_n \in \{I_n,-I_n\}$.

By assumption \ref{A2}, we can write $c_j^2=C_1^2j^{2\tau}(1+o_{j,n})$ for $N_0\le j\le n$, where $o_{j,n}=o_j(1)$ as $j\to \infty$.

\begin{lemma}\label{l.kn}
Let
    \[
    \tau^0_n:=\max_{\log(\log n)\le j\le n} |o_{j,n}|+a_n^{\tau+1/2} + a_n.
    \]
Then it holds uniformly for $x\in I_n^2$ that
    \[
    k_n(x)=\frac{C_1^2\Gamma(2\tau+1)}{(1-x)^{2\tau+1}} (1+O(\tau^0_n))
    \]
and 
    \[
     k_n(-x)=O(\tau^0_n)k_n(x),
    \]
where the implicit constants admit the same potential dependencies as in Theorem \ref{t.maingauss}.
\end{lemma}
\begin{proof}
Clearly, $\tau^0_n=o(1)$. By scaling invariance, we may assume $C_1=1$. For $x\in I_n^2$, we have
\begin{align*}
    k_n(x)&=\sum_{j=0}^{N_0-1}c_j^2x^j+ \sum_{j=N_0}^nj^{2\tau}(1+o_{j,n})x^j\\
    &=\sum_{j=0}^{N_0-1}c_j^2x^j+(1+O(\tau^0_n))\sum_{j=1}^nj^{2\tau}x^j+O\bigg(\sum_{j=1}^{\lfloor\log(\log n)\rfloor}j^{2\tau}x^j\bigg)\\
    &=: h(x)+(1+O(\tau^0_n))v_n(x)+O(t_n(x)).
\end{align*}
It is clear that $h(x)$ is uniformly bounded on any compact subset of $\mathbb R$, and these bounds are independent of $n$. For $O(t_n(x))$, we note that
    \[
    \left|t_n(x)\right| = O\left([\log(\log n)]^{2\tau+1}\right) =\frac{O(a_n^{\tau+1/2})}{(1-x)^{2\tau+1}}=\frac{O(\tau^0_n)}{(1-x)^{2\tau+1}},\quad x\in I_n^2.
    \]
The estimate for the middle term is based on the asymptotics of $v_n(x):=\sum_{j=1}^nj^{2\tau}x^j$. For $|x|<1$, $v_n(x)$ converges to $v_{\infty}(x)=\Li_{-2\tau}(x)$ as $n\to \infty$, where $\Li_{s}(z)$ is the polylogarithm function defined by
    \[
    \Li_{s}(z)=\sum_{j=1}^\infty \frac{z^j}{j^s},\quad |z|<1.
    \]
It is well known that (see \cite[p. 149]{T})
    \[
    \Li_s(z)=\Gamma(1-s)(-\log z)^{s-1}+\sum_{m=0}^\infty\zeta(s-m)\frac{(\log z)^m}{m!},
    \]
for $|\log z|<2\pi$ and $s\notin \{1,2,3,...\}$, where $\zeta(s)$ is the Riemann zeta function.
Thus, uniformly for $x\in I_n^2$ (and one could also let $x\in I_n^4$), we have 
\begin{align*}
    \Li_{-2\tau}(x) &=\Gamma(2\tau+1)(-\log x)^{-2\tau-1}+O(1)\\
    &=\frac{\Gamma(2\tau+1)}{(1-x)^{2\tau+1}}(1+O(\tau^0_n)),
\end{align*}
where in the second estimate, we implicitly used the fact that $1-x=O(a_n)=O(\tau^0_n)$. 

To estimate $\Li_{-2\tau}(-x)$ for $x\in I_n^2$, we use the first estimate of the last display and the duplication formula (see \cite[\S7.12]{L}),
    \[
    \Li_{-2\tau}(-x)+\Li_{-2\tau}(x)=2^{2\tau+1}\Li_{-2\tau}(x^2),
    \]
and find that uniformly for $x\in I_n^2$, 
    \[
    \Li_{-2\tau}(-x)=O(1).
    \]

For $|x|\in I_n^2$, we have 
    \[
    1-|x|\ge b_n=d_n/n\quad \text{and}\quad  
    |x|^{n+1}\le (1-b_n)^{2(n+1)} =O(e^{-2d_n}),
    \]
so
\begin{align*}
    \bigg|\sum_{j=n+1}^\infty j^{2\tau}x^{j}\bigg|&=|x|^{n+1}\bigg|\sum_{j=0}^{\infty}(j+n+1)^{2\tau}x^{j}\bigg|\\
    &= |x|^{n+1}O\bigg(\sum_{j=0}^{\infty}j^{2\tau}|x|^{j}+(n+1)^{2\tau}\sum_{j=0}^{\infty}|x|^{j}\bigg)\\
    &= |x|^{n+1}O\left(\Li_{-2\tau}(|x|)+\frac{(n+1)^{2\tau}}{1-|x|}\right)\\
    &=o(e^{-d_n}).
\end{align*}
Thus, uniformly for $x\in I_n^2$, 
    \[
    v_n(x)=\Li_{-2\tau}(x)-\sum_{j=n+1}^\infty j^{2\tau}x^{j} = \frac{\Gamma(2\tau+1)}{(1-x)^{2\tau+1}}(1+O(\tau^0_n)),
    \]
and 
    \[
    v_n(-x)=O(1).
    \]
Therefore, uniformly for $x\in I_n^2$,
    \[
    k_n(x)=\frac{C_1^2\Gamma(2\tau+1)}{(1-x)^{2\tau+1}} (1+O(\tau^0_n)).
    \]
Since
    \[
    k_n(-x)=h(-x)+C_1^2\sum_{j=N_0}^{n}j^{2\tau}(-x)^j+C_1^2\sum_{j=N_0}^n j^{2\tau}o_{j,n}(-x)^{j},
    \]
it follows that 
    \[
    |k_n(-x)| \le O(1)+ O(1)+O(\tau^0_n v_n(x))=O(\tau^0_n)k_n(x),\quad x\in I_n^2,
    \]
where, in the last estimate, we used $a_n^{2\tau+1} = O(\tau_n^0)$. This completes the proof of the lemma. 
\end{proof}

The proof extends to $k^{(i)}_n$, with $\tau$ replaced by $\tau+\frac 1 2 i$. Indeed, for $i\ge 1$, we have 
    \[
    k_n^{(i)}(x)=\sum_{j=i}^n c_j^2j(j-1)\cdots (j-i+1) x^{j-i}.
    \]
Recall that 
    \[
    c_j^2 = \begin{cases}
    C_1^2j^{2\tau}(1+o_{j,n})&\text{if}\quad N_0\le j\le n,\\
    O(1)&\text{if}\quad 0\le j<N_0.
    \end{cases}
    \]
Thus, for fixed $i\ge 1$ and all $j\ge i$, we have
    \[
    |c_j^2 j(j-1)\dots(j-i+1)|= \begin{cases}
    C_1^2 j^{2\tau+i} \left(1+O(o_{n,j})+O\left(1/j\right)\right)&\text{if}\quad j\ge \max\{N_0, i\},\\
    O(1)&\text{if}\quad i\le j < N_0.
    \end{cases}
    \]
Since $(a_n)^c = O(\frac 1{\log\log n})$ for any positive constant $c$,  we then let
\begin{equation}\label{e.taun}
    \tau_n:=\max_{\log(\log n)\le j\le n} \bigg|\frac{|c_j|}{C_1j^\tau}-1\bigg| +\frac 1 {\log\log n} = o(1),
\end{equation}
and obtain the following corollary.

\begin{corollary}\label{c.kn-der}
For any $0\le i\le 4$, it holds uniformly for $x\in I_n^2$ that
    \[
    k^{(i)}_n(x)=\frac{C^2_1\Gamma(2\tau+i+1)}{(1-x)^{2\tau+i+1}} (1+O(\tau_n))
    \]
and 
    \[
    k^{(i)}_n(-x)=O(\tau_n)k^{(i)}_n(x),
    \]
where the implicit constants have the same possible dependence given in Theorem~\ref{t.maingauss}.
\end{corollary}

\subsection{The pseudo-hyperbolic distance on the unit disk} 
This subsection reviews the pseudo-hyperbolic metric on the unit disk, which plays a central role in estimating $r_n(x,y)$ and its partial derivatives.

For $(x,y)\in \mathbb R\times \mathbb R$ with $1-xy\ne 0$, let us introduce the function
    \[
    \alpha:=\alpha(x,y):=1-\left(\frac{y-x}{1-xy}\right)^2=\frac{(1-x^2)(1-y^2)}{(1-xy)^2}.
    \]
Clearly, $0\le \alpha \le 1$. It is well known in complex analysis that
    \[
    \varrho(z,w):=\frac{|z-w|}{|1-\overline{w}z|}
    \]
defines a metric on the hyperbolic disk $\mathbb D :=\{z\in \mathbb C: |z|<1\}$ and is known as the pseudo-hyperbolic distance on $\mathbb D$  (see, e.g., \cite{G1}). A related notion is
    \[
    \frac{|x-y|^2}{(1-|x|^2)(1-|y|^2)} \equiv \frac 1{\alpha}-1,
    \]
which can be naturally extended to $\mathbb R^n$, where it is an isometric invariant for the conformal ball model $\mathbb B^n:=\{x\in \mathbb R^n:\|x\|<1\}$, and the classical Poincar\'e metric on $\mathbb B^n$ can also be computed from this invariant (see, e.g., \cite[\S 4.5]{R}). 

We first prove a property of the pseudo-hyperbolic distance that will be convenient later.

\begin{lemma}\label{l.ez} Let $0\le c<\frac 1{\sqrt 5}$ be a fixed constant. Suppose that for some $x,y \in (-1,1)$ with the same sign, we have
    \[
    \varrho(x,y)  \le c.
    \] 
Then, for every $z_1,z_2,z_3,z_4$ between $x$ and $y$, it holds that
    \[
    \frac{1}{1-z_1 z_2}=\frac{1+O(\varrho(x,y))}{1-z_3z_4},
    \]
where the implicit constants may depend on $c$. Consequently,
    \[
    \varrho(z_1,z_2)\le \varrho(x,y)[1+O(\varrho(x,y))].
    \]
\end{lemma}
\begin{proof}
It is clear that in the conclusion of Lemma~\ref{l.ez}, the second desired estimate follows immediately from the first desired estimate and the inequality $|z_1-z_2|\le |x-y|$. Below, we prove the first estimate.

Without loss of generality, we may assume $|x|\le |y|$.
Since $z_1,z_2,z_3,z_4$ are of the same sign, we have $x^2\le z_1z_2,z_3z_4\le y^2$, thus it suffices to show that
    \[
    \frac 1{1-y^2}  =(1+O(\varrho(x,y)) \frac 1{1-x^2}. 
    \]
It follows from the given hypothesis that $\alpha=1-\varrho^2(x,y) \ge 1-c^2$. Consequently,
    \[
    \frac{|x-y|}{\sqrt{(1-x^2)(1-y^2)}} =  \sqrt{\frac 1{\alpha }-1} \le \frac{\varrho(x,y)}{\sqrt{1-c^2}}.
    \]
Therefore,
\begin{align*}
    0\le \frac 1{1-y^2}-\frac 1{1-x^2} =\frac{y^2-x^2}{(1-x^2)(1-y^2)} 
    &\le  \frac{2|x-y|}{(1-x^2)(1-y^2)} \\
    &\le \frac{2\varrho(x,y)/\sqrt{1-c^2}}{\sqrt{(1-x^2)(1-y^2)}}\\
    &\le \frac{2\varrho(x,y)/\sqrt{1-c^2}}{1-y^2}.
\end{align*}
Since $\frac{2\varrho(x,y)}{\sqrt{1-c^2}} \le \frac{2c}{\sqrt{1-c^2}}<1$, we obtain
    \[
    \frac 1{1-y^2} \le \bigg(1-\frac{2\varrho(x,y)}{\sqrt{1-c^2}}\bigg)^{-1}\frac 1{1-x^2} = (1+O(\varrho(x,y)))\frac 1{1-x^2},
    \]
and the lemma follows.
\end{proof}

\subsection{Asymptotic estimates for the correlator and its partial derivatives} 
This subsection establishes asymptotic estimates for the normalized correlator $r_n(x,y)$ and its partial derivatives. Under hypothesis \ref{A2}, the polynomial $\widetilde{P}_n$ closely resembles a hyperbolic random polynomial. It is well known that the zero sets of complex Gaussian hyperbolic polynomials are asymptotically invariant under isometries of the hyperbolic disk $\mathbb D$. More precisely, as shown in \cite[Proposition 2.3.4]{HKPV09}, if $a_j$ are independent standard complex Gaussian random variables, then the zero set of the Gaussian hyperbolic series
    \[
    \widetilde{P}_{\infty,L}(z)=\sum_{j=0}^\infty \sqrt{\frac{L(L+1)\dots (L+j-1)}{j!}} a_j z^j
    \]
is invariant (in distribution) under the group of transformations
    \[
    \varphi_{a,b}(z)=\frac{az+b}{\overline{b}z+\overline{a}},\quad z\in \mathbb D,
    \]
where $a, b \in \mathbb C$ and $|a|^2-|b|^2=1$. These linear fractional transformations bijectively map $\mathbb D$ to itself and preserve the hyperbolic metric $ds^2=\frac{dx^2+dy^2}{(1-|z|^2)^2}$ and the hyperbolic area measure $\frac{dm(z)}{(1-|z|^2)^2}$. Thus, it seems natural to expect that asymptotic estimates for the correlation functions of $\widetilde{P}_n$ will involve isometry-invariant quantities, such as the pseudo-hyperbolic distance. The next few lemmas make this heuristic precise.

\begin{lemma} \label{lem.r}
Let $\tau_n$ be defined by \eqref{e.taun}. It holds uniformly for $(x,y)\in S_n\times S_n$ that 
\begin{equation}\label{e.r}
    r_n(x,y)=\alpha^{\tau+1/2}(1+O(\tau_n))
\end{equation}
    and 
\begin{equation} \label{e.1-r2}
    1-r_n^2(x,y)=(1-\alpha^{2\tau+1})(1+O(\sqrt[4]{\tau_n})).
\end{equation}
\end{lemma}
\begin{proof}
Inequality \eqref{e.r} follows immediately from \eqref{r} and Corollary~\ref{c.kn-der}.

To prove \eqref{e.1-r2}, we consider two cases according to whether $x$ and $y$ are close with respect to the pseudo-hyperbolic distance $\varrho(x,y)=|\frac{y-x}{1-xy}|$. 

Let $D:=\{(x,y)\in S_n\times S_n: \varrho(x,y)> \sqrt[4]{\tau_n}\}$ and $D':=(S_n\times S_n) \setminus D$ (see Figure \ref{fig4}).

\begin{figure}[htbp]
\centering
\includegraphics[scale=1]{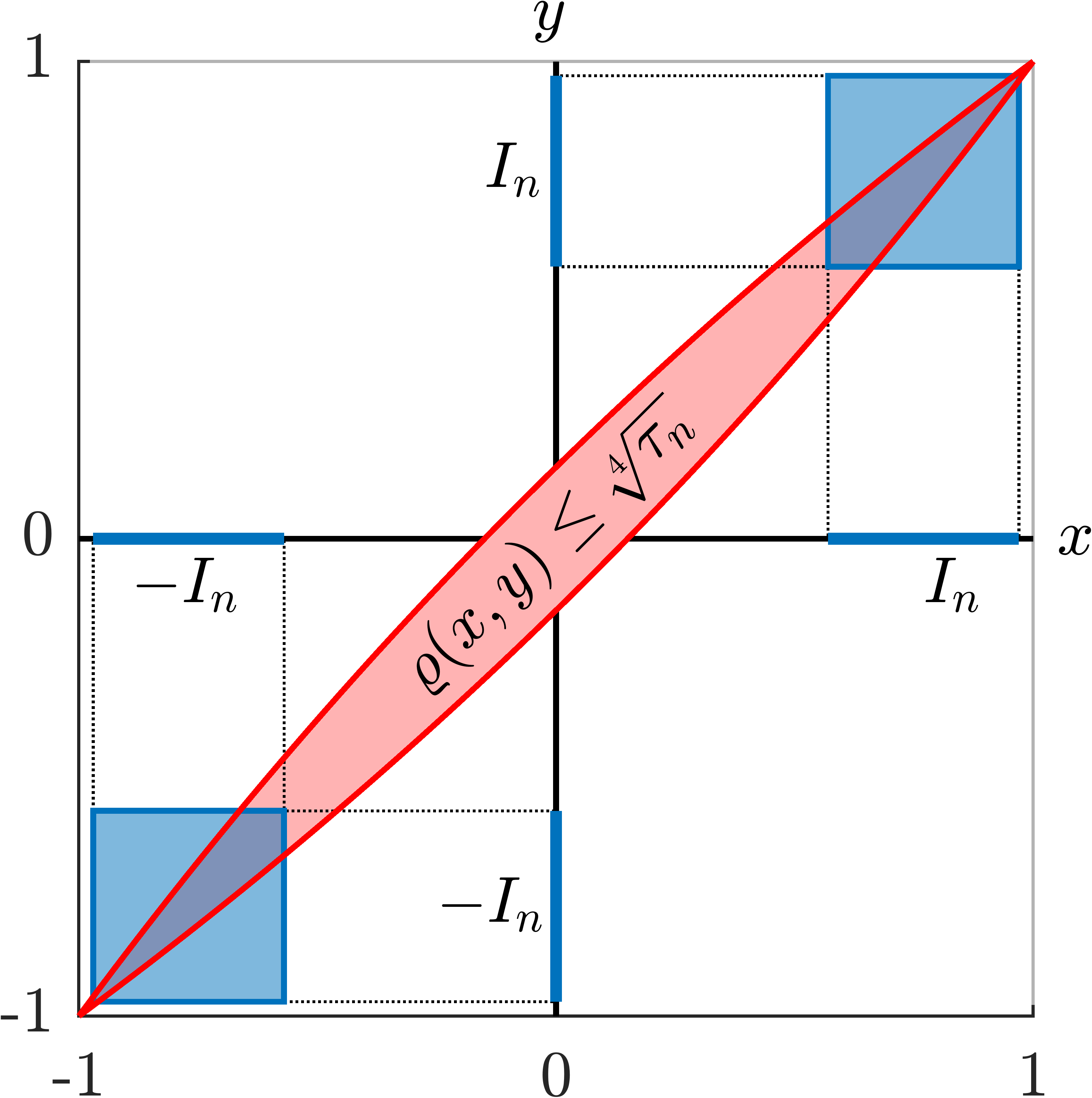}
\caption{$D'$ is the intersection of the red region with one of the two blue squares.}
\label{fig4}
\end{figure}

First, if $(x,y)\in D$, then $1- \alpha  > \sqrt{\tau_n}$, and so by the mean value theorem, we have
    \[
    1-\alpha^{2\tau+1}  \ge (2\tau+1)(1-\alpha)\min(\alpha^{2\tau},1) \ge (2\tau+1)\sqrt{\tau_n}\alpha^{2\tau+1}.
    \]
Therefore, we can use \eqref{e.r} to obtain
    \[
    1-r_n^2(x,y)=1-\alpha^{2\tau+1}(1+O(\tau_n))=(1-\alpha^{2\tau+1})(1+O(\sqrt{\tau_n})),
    \]  
which implies \eqref{e.1-r2}. 

We now assume that $(x,y)\in D'$. We have $\alpha=1+O(\sqrt{\tau_n})$. Therefore, using the mean value theorem, we obtain
\begin{equation}\label{e.alp}
    1-\alpha^{2\tau+1} = (2\tau+1)(1-\alpha)(1+O(\sqrt{\tau_n})).
\end{equation}
Using Lemma~\ref{l.ez},  for any $z_1,z_2$ between $x$ and $y$, we have
\begin{equation} \label{ez}
    \frac{1}{1-z_1 z_2}=\frac{1+O(\sqrt[4]{\tau_n})}{1-y^2}=\frac{1+O(\sqrt[4]{\tau_n})}{1-x^2}, \quad \text{and} \quad \frac{|z_1-z_2|}{|1-z_1z_2|}=O(\sqrt[4]{\tau_n}).
\end{equation}
Fix $x$. We then have
    \[
   1-r_n^2(x,y)=\frac{A(y)}{k_n(x^2)k_n(y^2)},
    \]
where $A(y):=k_n(x^2)k_n(y^2)-k_n^2(xy)$ for $y\in S_n$. Using Corollary~\ref{c.kn-der} and \eqref{ez}, we find
    \[
    k_n(x^2)k_n(y^2) = \left(\frac{C_1^2 \Gamma(2\tau+1)}{(1-x^2)^{2\tau+1}}\right)^2\left(1+O(\sqrt[4]{\tau_n})\right).
    \]
Clearly, $A(x)=0$, $A'(x)=0$, and for any $z$ between $x$ and $y$, we have
    \[
    A''(z)=2k_n(x^2)k_n'(z^2)+4z^2k_n(x^2)k_n''(z^2)-2x^2k_n''(xz)k_n(xz)-2x^2(k_n'(xz))^2.
    \]
Using Corollary~\ref{c.kn-der} and \eqref{ez}, we obtain
\begin{align*}
    A''(z) &= \frac{C_1^4\Gamma(2\tau+1)\Gamma(2\tau+2)}{(1-x^2)^{4\tau+4}} \Big[2+ 4(2\tau+2)(z^2-x^2) + O(\sqrt[4]{\tau_n}) \Big]\\
    &=\frac{2C_1^4\Gamma(2\tau+1)\Gamma(2\tau+2)}{(1-x^2)^{4\tau+4}}  (1+O(\sqrt[4]{\tau_n})).
\end{align*}
Since $A(x)=A'(x)=0$, according to the mean value theorem, there exists some $z$ between $x$ and $y$ such that 
    \[
    A(y)=\frac{1}{2}A''(z)(y-x)^2.
    \]
  Consequently, 
\begin{align*}
    1-r_n^2(x,y)=\frac{A(y)}{k_n(x^2)k_n(y^2)}=\frac{(y-x)^2A''(z)}{2k_n(x^2)k_n(y^2)}
    &=(2\tau+1)\bigg(\frac{y-x}{1-x^2}\bigg)^2(1+O(\sqrt[4]{\tau_n}))\\ 
    &= (2\tau+1)\bigg(\frac{y-x}{1-xy}\bigg)^2(1+O(\sqrt[4]{\tau_n})),
\end{align*}
which gives \eqref{e.1-r2} when combined with \eqref{e.alp}.
\end{proof}

Next, we estimate the partial derivatives of $r_n(x,y)$. To simplify computations and avoid messy algebra, let us introduce the function $\ell_n(x,y):=\log |r_n(x,y)|$. Note that
\begin{equation}\label{e.r2l}
    \frac{\partial r_n}{\partial x}(x,y)=r_n(x,y)\frac{\partial \ell_n}{\partial x}(x,y),\quad \frac{\partial r_n}{\partial y}(x,y)=r_n(x,y)\frac{\partial \ell_n}{\partial y}(x,y),
\end{equation}
and 
\begin{equation}\label{e.rxy2lxy}
    \frac{\partial^2 r_n}{\partial x\partial y}(x,y)=r_n(x,y)\left(\frac{\partial^2 \ell_n}{\partial x\partial y}(x,y)+\frac{\partial \ell_n}{\partial x}(x,y)\frac{\partial \ell_n}{\partial y}(x,y)\right).
\end{equation}
The following lemma indicates that one can take the natural log of \eqref{e.r} and then differentiate, and the estimates remain essentially valid.

\begin{lemma} \label{lem.l} It holds uniformly for $(x,y)\in S_n\times S_n$ that 
\begin{equation} \label{lx}
    \frac{\partial \ell_n}{\partial x}(x,y)=\frac{2\tau+1}{1-x^2}\bigg(\frac{y-x}{1-xy}\bigg)\left(1+O(\sqrt{\tau_n})\right)
\end{equation}
and 
\begin{equation} \label{lxy}
    \frac{\partial^2 \ell_n}{\partial x\partial y}(x,y)=\frac{2\tau+1}{(1-xy)^2}\left(1+O(\tau_n)\right).
\end{equation}
\end{lemma}
\begin{proof}
We start with the proof of \eqref{lx}. From \eqref{r}, we have
    \[
    \ell_n(x,y)=\log|k_n(xy)|-\frac 12 \log k_n(x^2)-\frac 12 \log k_n(y^2),
    \]
which leads to
\begin{equation}\label{e.lx}
    \frac{\partial \ell_n}{\partial x}(x,y)
    =y\frac{k_n'(xy)}{k_n(xy)}-x\frac{k_n'(x^2)}{k_n(x^2)}=\frac{yk_n'(xy)k_n(x^2)-xk_n'(x^2)k_n(xy)}{k_n(xy)k_n(x^2)}.
\end{equation}
Utilizing Corollary~\ref{c.kn-der}, it holds uniformly for $(x,y)\in S_n\times S_n$ that
\begin{equation} \label{dk-over-k}
    \frac{k_n'(xy)}{k_n(xy)}=\frac{2\tau+1}{1-xy}(1+O(\tau_n)) \quad \text{and}\quad \frac{k_n'(x^2)}{k_n(x^2)}=\frac{2\tau+1}{1-x^2}(1+O(\tau_n)).
\end{equation}
We now divide the proof into two cases, similar to the proof of Lemma~\ref{lem.r}.

Let $D:=\{(x,y)\in S_n\times S_n: \varrho(x,y) \ge \sqrt{\tau_n}\}$ and $D':=(S_n\times S_n) \setminus D.$

For $(x,y)\in D$,  one observes that
    \[
    \frac {1-x^2}{|y-x|} = \frac {1-xy}{|y-x|} + x\frac{y-x}{|y-x|} < \frac 2{\sqrt{\tau_n}},
    \]
thus yielding
    \[
    \max\left\{\frac{\tau_n}{1-x^2},\frac{\tau_n}{1-xy}\right\} < 2\sqrt{\tau_n}\frac{|y-x|}{(1-x^2)(1-xy)}.
    \]
Hence, \eqref{dk-over-k} implies 
    \[
    \frac{\partial \ell_n}{\partial x}(x,y)=y\frac{k_n'(xy)}{k_n(xy)}-x\frac{k_n'(x^2)}{k_n(x^2)}=\frac{2\tau+1}{1-x^2}\left(\frac{y-x}{1-xy}\right)\left(1+O(\sqrt{\tau_n})\right).
    \]

We now suppose that $(x,y)\in D'$, then $\alpha=1+O(\tau_n)$. Using Lemma~\ref{l.ez}, for all $z_1,z_2$ between $x$ and $y$, we have
\begin{equation} \label{ez'}
    \frac{1}{1-z_1 z_2}=\frac{1+O(\sqrt{\tau_n})}{1-y^2}=\frac{1+O(\sqrt{\tau_n})}{1-x^2}, \quad \text{and}\quad \frac{|z_1-z_2|}{|1-z_1z_2|}=O(\sqrt{\tau_n}).
\end{equation}
Fix $x\in S_n$ and express $\frac{\partial \ell_n}{\partial x}(x,y)$ as
    \[
    \frac{\partial \ell_n}{\partial x}(x,y) = \frac{B(y)}{k_n(xy)k_n(x^2)},
    \]
where $B(y):=yk_n'(xy)k_n(x^2)-xk_n'(x^2)k_n(xy)$, treated as a function of $y\in S_n$. Then, $B(x)=0$ and 
    \[
    B'(y)=k_n(x^2)[k_n'(xy)+xyk_n''(xy)]-x^2k_n'(x^2)k_n'(xy).
    \]
By employing Corollary~\ref{c.kn-der} and \eqref{ez'}, we find that for any $z$ between $x$ and $y$, 
\begin{align*}
    \frac{B'(z)}{k_n(x^2)k_n(xz)}&=\frac{2\tau+1}{(1-x^2)^2}\Big[(1-x^2) +xz (2\tau+2) - x^2 (2\tau+1) + O(\sqrt{\tau_n})\Big]\\
    &=\frac {2\tau+1}{(1-x^2)^2}\left[1+(2\tau+2)x(z-x)+O(\sqrt{\tau_n}) \right]\\
    &=\frac {2\tau+1}{(1-x^2)^2}\left[1+O(\sqrt{\tau_n})\right].
\end{align*}
Next, using the mean value theorem and \eqref{ez'}, we see that for some $z$ between $x$ and $y$,
    \[
    B(y)=B(x)+B'(z)(y-x)=B'(z)(y-x),
    \]
thus yielding
\begin{align*}
    \frac{\partial \ell_n}{\partial x}(x,y)=\frac{(y-x)B'(z)}{k_n(xy)k_n(x^2)} &= \frac{(2\tau+1)(y-x)}{(1-x^2)^2}\left(1+O(\sqrt{\tau_n})\right)\\
    &=\frac{2\tau+1}{1-x^2}\left(\frac{y-x}{1-xy}\right)\left(1+O(\sqrt{\tau_n})\right), 
\end{align*}
and \eqref{lx} is proved.

To establish \eqref{lxy}, we utilize Corollary \ref{c.kn-der}, leading to
\begin{align}
    \label{e.lxy-expand} \frac{\partial^2 \ell_n}{\partial x\partial y}(x,y)&=\frac{k_n'(xy)}{k_n(xy)}+xy\frac{k''_n(xy)}{k_n(xy)}- xy\left(\frac{k'_n(xy)}{k_n(xy)}\right)^2\\
    \nonumber &=\frac{2\tau+1}{1-xy}+xy\frac{(2\tau+1)(2\tau+2)}{(1-xy)^2}-xy\left(\frac{2\tau+1}{1-xy}\right)^2 + O\left(\frac {\tau_n}{(1-xy)^2}\right)\\
    \nonumber &=\frac{2\tau+1}{(1-xy)^2}\left(1+O(\tau_n)\right),
\end{align}
thereby concluding the proof.
\end{proof}

\subsection{Asymptotics of the one-point correlation function}
As a consequence of the preceding estimates, we derive an asymptotic expression for the one-point correlation function $\rho_n^{(1)}$.

\begin{corollary} 
Uniformly for $x\in S_n \cup (-S_n)$, it holds that
\begin{equation} \label{rho1-in}
    \rho_n^{(1)}(x)=\frac{1}{\pi}\frac{\sqrt{2\tau+1}}{1-x^2}\left(1+O(\tau_n)\right).
\end{equation}
\end{corollary}

We note that a variant of \eqref{rho1-in} is also implicit in \cite{DNV}, with a stronger bound for the error term (but more stringent assumptions on the coefficients $c_j$).

\begin{proof} 
By symmetry, it suffices to consider $x\in S_n$. Using \eqref{e.rxy2lxy}, $r_n(x,x)=1$, and $\frac{\partial \ell_n}{\partial x}(x,x)=\frac{\partial \ell_n}{\partial y}(y,y)=0$, we see that
    \[
    \rho^{(1)}_n(x)=\frac{1}{\pi}\sqrt{\frac{\partial^2r_n}{\partial x\partial y}(x,x)}=\frac{1}{\pi}\sqrt{\frac{\partial^2\ell_n}{\partial x\partial y}(x,x)}.
    \]
Thus, \eqref{rho1-in} follows immediately from \eqref{lxy}.
\end{proof}

\subsection{Asymptotics of the two-point correlation function}
We now establish asymptotic estimates for the two-point correlation function $\rho_n^{(2)}$. Recall that 
\begin{equation} \label{rho2}
    \rho^{(2)}_n(x,y)=\frac{1}{\pi^2}\left(\sqrt{1-\delta^2_n(x,y)}+\delta_n(x,y) \arcsin \delta_n(x,y)\right)\frac{\sigma_n(x,y)}{\sqrt{1-r_n^2(x,y)}},
\end{equation}
where, using \eqref{e.r2l} and \eqref{e.rxy2lxy}, we write $\sigma_n$ and $\delta_n$ as
\begin{equation} \label{sigma}
\begin{split}
    \sigma_n(x,y)&=\pi^2\rho_n^{(1)}(x)\rho_n^{(1)}(y)  \\
    &\quad \times \sqrt{\bigg(1-\frac{(r_n(x,y)\frac{\partial \ell_n}{\partial x}(x,y))^2}{(1-r_n^2(x,y))\frac{\partial^2\ell_n}{\partial x\partial y}(x,x)}\bigg)\bigg(1-\frac{(r_n(x,y)\frac{\partial \ell_n}{\partial y}(x,y))^2}{(1-r_n^2(x,y))\frac{\partial^2\ell_n}{\partial x\partial y}(y,y)}\bigg)}
\end{split}
\end{equation}
and 
\begin{equation} \label{delta}
\begin{split}
    \delta_n(x,y)&=\frac{r_n(x,y)}{\sigma_n(x,y)}\left(\frac{\partial^2\ell_n}{\partial x\partial y}(x,y)+\frac{\frac{\partial \ell_n}{\partial x}(x,y)\frac{\partial \ell_n}{\partial y}(x,y)}{1-r_n^2(x,y)}\right).
\end{split}
\end{equation}
To keep the proof from being too long, we separate the estimates into several lemmas.

\begin{lemma}\label{l.cor-asymp}
Uniformly for $(x,y)\in S_n\times S_n$, it holds that
\begin{equation} \label{rho2-in}
    \rho^{(2)}_n(x,y) = \frac{2\tau+1}{\pi^2}\frac{1+f_{\tau}(\alpha)}{(1-x^2)(1-y^2)} \left(1+O(\sqrt[16]{\tau_n})\right), 
\end{equation}
where $f_\tau$ is defined as in \eqref{ftau}. Furthermore, there exists a constant $\alpha_0>0$ (independent of $n$, but possibly dependent on the implicit constants and rate of convergence in conditions \ref{A1} and \ref{A2})  such that when $\alpha \le \alpha_0$, the following holds
\begin{equation} \label{e.alphasmall-in}
    \rho_n^{(2)}(x,y)-\rho_n^{(1)}(x)\rho_n^{(1)}(y) = \frac{O(\alpha^{2\tau+1})}{(1-x^2)(1-y^2)}.
\end{equation}
\end{lemma}
\begin{proof} 
We begin with \eqref{rho2-in}. To establish this, we first derive asymptotic estimates for $\sigma_n(x,y)$ and $\delta_n(x, y)$.  

For $\sigma_n(x,y)$, we first show that 
\begin{equation} \label{e.condvar}
\begin{split}
    &(1-r_n^2(x,y))\frac{\partial^2\ell_n}{\partial x\partial y}(x,x)-\bigg(r_n(x,y)\frac{\partial \ell_n}{\partial x}(x,y)\bigg)^2\\
    &=\frac{2\tau+1}{(1-x^2)^2}\left(1-\alpha^{2\tau+1}-(2\tau+1)(1-\alpha)\alpha^{2\tau+1}\right)\left(1+O(\sqrt[16]{\tau_n})\right).
\end{split}
\end{equation}
Define $g(u)=1- u^{2\tau+1}-(2\tau+1)(1-u)u^{2\tau+1}$ for $u\in [0,1]$. It is evident that $g$ is non-increasing on $[0,1]$ and $g(1)=0$.
Hence, $g(\alpha)\ge 0$, which implies
    \[
    1-\alpha^{2\tau+1} \ge (2\tau+1)(1-\alpha)\alpha^{2\tau+1}. 
    \]

We consider two cases. First, if $\sqrt[8]{\tau_n} (1-\alpha^{2\tau+1}) \le g(\alpha)$, then \eqref{e.condvar} follows immediately from Lemma~\ref{lem.r} and Lemma~\ref{lem.l}. Now, if $\sqrt[8]{\tau_n} (1-\alpha^{2\tau+1}) > g(\alpha)$, we will show that $x$ and $y$ are close in the pseudo-hyperbolic distance; namely,
    \[
    \varrho(x,y)=\frac{|x-y|}{|1-xy|}=O(\sqrt[16]{\tau_n}).
    \]
To see this, note that on $[0,1)$ the inequality $g(u)-\sqrt[8]{\tau_n}(1-u^{2\tau+1}) \le 0$ implies
    \[
    u^{2\tau+1} \ge (1-\sqrt[8]{\tau_n})\frac {1-u^{2\tau+1}}{(2\tau+1)(1-u)} \ge (1-\sqrt[8]{\tau_n})\min(1,u^{2\tau}),
    \]
thanks to the mean value theorem. This further implies
    \[
    u\ge \min\left(1-\sqrt[8]{\tau_n}, (1-\sqrt[8]{\tau_n})^{1/(2\tau+1)}\right) =1-O(\sqrt[8]{\tau_n}),
    \]
thus yielding $\alpha=1+O(\sqrt[8]{\tau_n})$ as desired.

Now, since $g(1)=g'(1)=0$ and $g''(t)=  (2\tau+2)(2\tau+1)(1+O(\sqrt[8]{\tau_n}))$ for every $t\in [\alpha,1]$, we can apply the mean value theorem to rewrite the right-hand side (RHS) of \eqref{e.condvar} as
\begin{equation} \label{eq3.22}
    {\rm RHS}=\frac{(2\tau+1)^2(\tau+1)}{(1-x^2)^2}(1-\alpha)^2(1+O(\sqrt[8]{\tau_n})). 
\end{equation}
Fix $x$. We can rewrite the left-hand side of \eqref{e.condvar} as
    \[
    (1-r_n^2(x,y))\frac{\partial^2\ell_n}{\partial x\partial y}(x,x)-\bigg(r_n(x,y)\frac{\partial \ell_n}{\partial x}(x,y)\bigg)^2=\frac{\frac 12 A(y)A''(x)-B^2(y)}{k_n^3(x^2)k_n(y^2)},
    \]
recalling that $A(y)=k_n(x^2)k_n(y^2)-k_n^2(xy)$ and $B(y)=yk_n'(xy)k_n(x^2)-xk_n'(x^2)k_n(xy)$. Let 
    \[
    C(y):= \frac 12 A(y)A''(x)-B^2(y).
    \]
We check at once that 
\begin{align*}
    C(x)&=C'(x)=0,\\
    C''(x)&=\frac{1}{2}[A''(x)]^2-2[B'(x)]^2=0,\\
    C'''(x)&=\frac 12 A''(x)[A'''(x)-6B''(x)]=0,
\end{align*}
and for all $z$ between $x$ and $y$,
\begin{align*}
    C^{(4)}(z)&=\frac 12 A^{(4)}(z)A''(x)-2B(z)B^{(4)}(z)-8B'(z)B'''(z)-6[B''(z)]^2.
\end{align*}
By Lemma~\ref{l.ez}, it follows that for all $z_1,z_2$ between $x$ and $y$, we have
\begin{equation} \label{e.ez-cor}
    \frac{1}{1-z_1 z_2}=\frac{1+O(\sqrt[16]{\tau_n})}{1-y^2}=\frac{1+O(\sqrt[16]{\tau_n})}{1-x^2}, \quad\text{and}\quad \frac{|z_1-z_2|}{1-z_1z_2|}=O(\sqrt[16]{\tau_n}).
\end{equation}
Using Corollary~\ref{c.kn-der} and \eqref{e.ez-cor}, we obtain
    \[
    k_n^3(x^2)k_n(y^2)=\left(\frac{C_1^2\Gamma(2\tau+1)}{(1-x^2)^{2\tau+1}}\right)^4 \left(1+O(\sqrt[16]{\tau_n})\right),
    \]
and for any $z$ between $x$ and $y$, by arguing as in the proof of Lemma~\ref{lem.r} and Lemma~\ref{lem.l}, we have
\begin{align*}
    B(z) &=\frac{C_1^4\Gamma(2\tau+1)\Gamma(2\tau+2)}{(1-x^2)^{4\tau+4}}(z-x)\Big(1+O(\sqrt[16]{\tau_n})\Big)\\
    &=\frac{C_1^4\Gamma(2\tau+1)\Gamma(2\tau+2)}{(1-x^2)^{4\tau+3}} O(\sqrt[16]{\tau_n}),
\end{align*}
and for $i\ge 1$,
\begin{align*}
    B^{(i)}(z) &= k_n(x^2)\Big[zk_n^{(i+1)}(xz)x^i+ik_n^{(i)}(xz)x^{i-1}\Big]-k'_n(x^2)k_n^{(i)}(xz)x^{i+1}\\
    &= \frac{C_1^4 \Gamma(2\tau+1)\Gamma(2\tau+i+1)}{(1-x^2)^{4\tau+i+3}}\\
    &\quad \times \left((2\tau+i+1)zx^i+  ix^{i-1}(1-x^2)-(2\tau+1)x^{i+1}+O(\sqrt[16]{\tau_n})\right)\\
    &= \frac{C_1^4 \Gamma(2\tau+1)\Gamma(2\tau+i+1)}{(1-x^2)^{4\tau+i+3}}\left(ix^{i-1} +O(\sqrt[16]{\tau_n})\right).
\end{align*}
Similarly,
\begin{align*}
    A''(x)&= \frac{2C_1^4\Gamma(2\tau+1)\Gamma(2\tau+2)}{(1-x^2)^{4\tau+4}}  (1+O(\sqrt[16]{\tau_n})),\\
    A^{(4)}(z)&= 16k_n^{(4)}(z^2)k_n(x^2)(1+O(\tau_n))-2k_n^{(4)}(xz)k_n(xz) -8k_n'''(xz)k_n'(xz)-6(k''_n(xz))^2 \\
    &= \frac{C_1^4 \Gamma(2\tau+5)\Gamma(2\tau+1)}{(1-x^2)^{4\tau+6}}\bigg[14-\frac{8(2\tau+1)}{2\tau+4}-\frac {6(2\tau+2)(2\tau+1)}{(2\tau+4)(2\tau+3)}+ O(\sqrt[16]{\tau_n})\bigg].
\end{align*}
Consequently,
    \[
    C^{(4)}(z) = \frac{24C_1^8(\tau+1)\Gamma^2(2\tau+2)\Gamma^2(2\tau+1)}{(1-x^2)^{8\tau+10}} (1+O(\sqrt[16]{\tau_n})).
    \]
    
\begin{remark}\label{r.alt-approach}
We may arrive at this estimate by formally differentiating the leading asymptotics of $C(y)$ (obtained using Corollary~\ref{c.kn-der}) with respect to $y$, and then letting $y=x$. In general, when $\varrho(x,y)$ is $o(1)$ small, there could be cancellation inside the differentiated asymptotics. In such cases, the expression obtained from the formal differentiation may no longer be the leading asymptotics for the underlying derivative of $C(y)$. Lemma~\ref{l.ez} is useful in examining the differentiated asymptotics, effectively allowing us to set $y=x$ at the cost of error terms of (theoretically) smaller orders.
\end{remark}

Now, applying the mean value theorem and \eqref{e.ez-cor}, we find
    \[
    \frac{C(y)}{k_n^3(x^2)k_n(y^2)}=\frac{C^{(4)}(z)\frac{(y-x)^4}{4!} }{k_n^3(x^2)k_n(y^2)}=\frac{(2\tau+1)^2(\tau+1)}{(1-x^2)^2}(1-\alpha)^2(1+O(\sqrt[16]{\tau_n})),
    \]
which gives \eqref{e.condvar} when combined with \eqref{eq3.22}.

From \eqref{e.condvar}, it holds uniformly for $(x,y)\in S_n\times S_n$ that
\begin{align*}
    1-\frac{(r_n(x,y)\frac{\partial \ell_n}{\partial x}(x,y))^2}{(1-r_n^2(x,y))\frac{\partial^2\ell_n}{\partial x\partial y}(x,x)}&=\bigg(1-\frac{(2\tau+1)(1-\alpha)\alpha^{2\tau+1}}{1-\alpha^{2\tau+1}}\bigg)\left(1+O(\sqrt[16]{\tau_n})\right).
\end{align*}
Likewise, 
\begin{align*}
    1-\frac{(r_n(x,y)\frac{\partial \ell_n}{\partial y}(x,y))^2}{(1-r_n^2(x,y))\frac{\partial^2\ell_n}{\partial x\partial y}(y,y)}=\bigg(1-\frac{(2\tau+1)(1-\alpha)\alpha^{2\tau+1}}{1-\alpha^{2\tau+1}}\bigg)\left(1+O(\sqrt[16]{\tau_n})\right).
\end{align*}
Combining the above estimates with \eqref{sigma}, we get
\begin{equation} \label{asig}
    \frac{\sigma_n(x,y)}{\pi^2\rho_n^{(1)}(x)\rho_n^{(1)}(y)} = \bigg(1-\frac{(2\tau+1)(1-\alpha)\alpha^{2\tau+1}}{1-\alpha^{2\tau+1}}\bigg)\left(1+O(\sqrt[16]{\tau_n})\right).
\end{equation}

To obtain the asymptotics for $\delta_n(x,y)$, we first show that
\begin{equation} \label{e.deltanum}
    \frac{\partial^2\ell_n}{\partial x\partial y}(x,y)+\frac{\frac{\partial \ell_n}{\partial x}(x,y)\frac{\partial \ell_n}{\partial y}(x,y)}{1-r_n^2(x,y)}
    =\frac{2\tau+1}{(1-x^2)(1-y^2)}\left(\alpha -\frac{(2\tau+1)(1-\alpha)}{1-\alpha^{2\tau+1}}\right)(1+O(\sqrt[16]{\tau_n})).
\end{equation}
The argument is similar to the proof of \eqref{e.condvar}, so we will only mention the key steps. We may assume $\alpha =1+O(\sqrt[8]{\tau_n})$, otherwise \eqref{e.deltanum} will follow from Lemma~\ref{lem.r} and Lemma~\ref{lem.l}. With this constraint on $\alpha$, we have
    \[
    \alpha -\frac{(2\tau+1)(1-\alpha)}{1-\alpha^{2\tau+1}} =-\frac 1 2\left(\frac{(\alpha-1)^2(2\tau+1)(2\tau+2)}{1-\alpha^{2\tau+1}}\right)(1+O(\sqrt[8]{\tau_n})).
    \]
Proceeding as before and applying Lemma~\ref{lem.r}, Lemma~\ref{l.ez}, and Corollary~\ref{c.kn-der}, it suffices to show
\begin{align*}
   \frac{\partial^2\ell_n}{\partial x\partial y}(x,y)(1-r_n^2(x,y))+\frac{\partial \ell_n}{\partial x}(x,y)\frac{\partial \ell_n}{\partial y}(x,y)
   = -\frac{(2\tau+1)^2(\tau+1)(y-x)^4}{(1-x^2)^4}(1+O(\sqrt[8]{\tau_n})). 
\end{align*}
Fix $x$. We then write the left-hand side as $\frac{E(y)}{k^2_n(xy)k_n(x^2)k_n(y^2)}$, where
\begin{align*}
    E(y) &:= A(y) a(y)+ B(y) b(y), \\
    a(y) &:= [k_n'(xy)+k_n''(xy)xy]k_n(xy)-xy[k_n'(xy)]^2,\\
    b(y) &:= xk_n'(xy)k_n(y^2)-yk_n'(y^2)k_n(xy).
\end{align*}
One can check that, as a function of $y$, $E(x)=E'(x)=E''(x)=E'''(x)=0$. Indeed, by direct computation,
\begin{align*}
    A(x)&=A'(x)=b(x)=B(x)=0,\\
    A''(x)&=-2b'(x)=2B'(x)=2a(x),\\
    A'''(x)&=-2b''(x)=6B''(x)=6a'(x),
\end{align*}
from which one can see that $E$ vanishes up to the third derivative at $y=x$.  Furthermore, using Lemma~\ref{l.ez} and Corollary~\ref{c.kn-der}, we can show that for all $z$ between $x$ and $y$,
    \[
    E^{(4)}(z)=- \frac{24(2\tau+1)^2(\tau+1)C_1^4 (\Gamma(2\tau+1))^4}{(1-x^2)^{8\tau+8}} (1+O(\sqrt[16]{\tau_n})),
    \]
which implies the desired estimate, thanks to the mean value theorem. 

In the following, we include some computations in the spirit of Remark~\ref{r.alt-approach}. For simplicity, let $D_1=C_1^2 \Gamma(2\tau+1)\Gamma(2\tau+2)$. Then
\begin{align*}
    E(y) &=\frac{D_1^2}{2\tau+1}\Big(\frac{1}{(1-x^2)^{2\tau+1}(1-y^2)^{2\tau+1}} - \frac{1}{(1-xy)^{4\tau+2}}\Big)\\
    &\qquad \times \Big(\frac{1}{(1-xy)^{4\tau+3}}+\frac{(2\tau+2)xy}{(1-xy)^{4\tau+4}} - \frac{(2\tau+1)xy}{(1-xy)^{4\tau+4}}\Big)\\
    &\quad + D_1^2\Big(\frac{y}{(1-xy)^{2\tau+2}(1-x^2)^{2\tau+1}} - \frac{x}{(1-x^2)^{2\tau+2}(1-xy)^{2\tau+1}}\Big)\\
    &\qquad \times \Big(\frac{x}{(1-xy)^{2\tau+2}(1-y^2)^{2\tau+1}} - \frac{y}{(1-y^2)^{2\tau+2}(1-xy)^{2\tau+1}}\Big)\\
    &\quad +\text{lower order terms}\\
    &=\frac{D_1^2}{2\tau+1}\Big(\frac{1}{(1-x^2)^{2\tau+1}(1-y^2)^{2\tau+1}} - \frac{1}{(1-xy)^{4\tau+2}}\Big)\frac{1}{(1-xy)^{4\tau+4}}  \\
    &\quad + D_1^2\Big(\frac{-(y-x)^2}{(1-xy)^{4\tau+4}(1-x^2)^{2\tau+2} (1-y^2)^{2\tau+2}}\Big)+\text{lower order terms},
\end{align*}
\begin{align*}
    \frac{d^k}{dy^k}\left(\frac 1{(1-y^2)^{2\tau+1}}\right) &=\frac{(2y)^k(2\tau+1)\dots(2\tau+k)}{(1-y^2)^{2\tau+k+1}}+ \text{lower order terms},\\
    \frac{d^k}{dy^k}\left(\frac 1{(1-xy)^{2\tau+1}}\right) &=\frac{x^k(2\tau+1)\dots(2\tau+k)}{(1-xy)^{2\tau+k+1}},
\end{align*}
and 
\begin{align*}
    E^{(4)}(y)&=\frac{D_1^22(\tau+1)}{(2\tau+1)(1-x^2)^{8\tau+10}}\Big(16(2\tau+1)(2\tau+3)(2\tau+4)\\
    &\qquad +32(2\tau+1) (2\tau+3)(4\tau+4)+24(2\tau+1)(4\tau+4)(4\tau+5)\\
    &\qquad + 8(2\tau+1)(2)(4\tau+5)(4\tau+6)+2(4\tau+5)(4\tau+6)(4\tau+7)\\
    &\qquad - 4(8\tau+6)(8\tau+7)(8\tau+9)- 24(2\tau+1)(4\tau+5)\\
    &\qquad -96(2\tau+1)(2\tau+2)-48(2\tau+1)(2\tau+3) \Big)+\text{lower order terms}\\
    &=\frac{(-24)D_1^2 (\tau+1)}{(1-x^2)^{8\tau+10}}+\text{lower order terms,}
\end{align*}
as claimed.

On account of \eqref{delta}, \eqref{asig}, and Lemma \ref{lem.r}, we conclude that
\begin{align*}
    \delta_n(x,y)&=\alpha^{\tau+1/2} \frac{\alpha (1-\alpha^{2\tau+1})-(2\tau+1)(1-\alpha)}{1-\alpha^{2\tau+1}-(2\tau+1)\alpha^{2\tau+1}(1-\alpha)}\left(1+O(\sqrt[16]{\tau_n})\right)\\
    &=\Delta_\tau(\alpha)\left(1+O(\sqrt[16]{\tau_n})\right),
\end{align*}
where $\Delta_\tau$ is defined by \eqref{del}. 
Let 
    \[
    \Lambda(\delta):=\sqrt{1-\delta^2}+\delta\arcsin \delta,\quad \delta\in [-1,1].
    \]
Since $\Lambda(\delta)\ge 1$ and $|\Lambda'(\delta)|=|\arcsin \delta|\le \pi/2$ for all $\delta\in (-1,1)$, and $|\Delta_\tau(\alpha)|\le 1$ for all $(x,y)\in S_n\times S_n$, it follows from the mean value theorem that
\begin{align}
    \notag \Lambda(\delta_n(x,y)) & = \Lambda(\Delta_\tau(\alpha))+O(|\Delta_\tau(\alpha)O(\sqrt[16]{\tau_n})|)\\
    \label{addel}  & = \Lambda(\Delta_\tau(\alpha))(1+O(\sqrt[16]{\tau_n})).  
\end{align}
Substituting \eqref{e.1-r2}, \eqref{asig}, and \eqref{addel} into \eqref{rho2}, we deduce \eqref{rho2-in} as claimed. 

We now discuss the proof of \eqref{e.alphasmall-in}. In the computation below, we assume that $\alpha\le \alpha_0$, where $\alpha_0$ is a sufficiently small positive constant.  Using Lemmas \ref{lem.r} and \ref{lem.l}, we have
    \[
    \frac{(r_n(x,y)\frac{\partial\ell_n}{\partial x}(x,y))^2}{(1-r_n^2(x,y)) \frac{\partial^2\ell_n}{\partial x\partial y}(x,x)} = \frac{O(\alpha^{2\tau+1})O\left(\frac 1 {(1-x^2)^2}\right)}{(1+O(\alpha^{2\tau+1}))\frac1{(1-x^2)^2}} = 1+O(\alpha^{2\tau+1}).
    \]
Similarly,
    \[
    \frac{(r_n(x,y)\frac{\partial \ell_n}{\partial y}(x,y))^2}{(1-r_n^2(x,y))\frac{\partial^2\ell_n}{\partial x\partial y}(y,y)} = 1+O(\alpha^{2\tau+1}).
    \]
Using \eqref{sigma}, it follows that
    \[
    \sigma_n(x,y)=\pi^2 \rho_n^{(1)}(x)\rho_n^{(1)}(y)(1+O(\alpha^{2\tau+1})).
    \]
From Lemma~\ref{lem.l}, we also have
\begin{align*}
    &\frac{\partial^2 \ell_n}{\partial x \partial y}(x,y) + \frac{\frac{\partial \ell_n}{\partial x}(x,y)\frac{\partial \ell_n}{\partial y}(x,y)}{1-r_n^2(x,y)}\\
    &= O\bigg(\frac1{(1-xy)^2}\bigg)+O\bigg(\frac{(x-y)^2}{(1-xy)^2(1-x^2)(1-y^2)}\frac 1{1+O(\alpha^{2\tau+1})}\bigg) \\
    &=O\bigg(\frac{1}{(1-x^2)(1-y^2)}\bigg).
\end{align*}
Thus, it follows from \eqref{delta} that
\begin{align*}
    \delta_n(x,y) = O(\alpha^{\tau+\frac 1 2}).
\end{align*}
Recall that $\Lambda(\delta)=\sqrt{1-\delta^2}+\delta \arcsin\delta$, which satisfies $\Lambda'(0)=0$ and $\Lambda''(\delta)=O(1)$ for $\delta$ near $0$. Thus, $\Lambda(\delta)=1+O(\delta^2)$ for $\delta$ near $0$. Consequently, using Lemma~\ref{lem.r} again, we obtain
\begin{align*}
    \frac{\rho_n^{(2)}(x,y)}{\rho_n^{(1)}(x)\rho_n^{(1)}(y)} &= \Lambda(\delta_n(x,y))\frac{\sigma_n(x,y)}{\pi^2\rho_n^{(1)}(x)\rho_n^{(1)}(y) \sqrt{1-r_n^2(x,y)}}\\
    &=(1+O(\delta_n(x,y)^2))(1+O(\alpha^{2\tau+1}))\frac 1{\sqrt {1+O(\alpha^{2\tau+1})}}\\
    &= 1+O(\alpha^{2\tau+1}).
\end{align*}
This completes the proof of Lemma~\ref{l.cor-asymp}.
\end{proof}

The next lemma describes the correlation between positive and negative roots. 

\begin{lemma} \label{l.cor-asymp-mix}
Uniformly for $(x,y)\in S_n\times S_n$, it holds that
\begin{equation} \label{rho2-in-mix}
    \rho^{(2)}_n(-x,y)-\rho^{(1)}_n(-x) \rho^{(1)}_n(y) = \frac{2\tau+1}{\pi^2}\frac{\alpha^{2\tau+1}}{(1-x^2)(1-y^2)} o(1).
\end{equation}
\end{lemma}
\begin{proof}
Using \eqref{r} and Corollary~\ref{c.kn-der}, we have
    \[
    r_n(-x,y)=\frac{k_n(-x,y)}{\sqrt{k_n(x^2)k_n(y^2)}}=\frac{O(\tau_n)k_n(x,y)}{\sqrt{k_n(x^2)k_n(y^2)}}=O(\tau_n)r_n(x,y)=o(1).
    \]
Now, by explicit computation (see also \eqref{e.lx} and \eqref{e.lxy-expand}),
\begin{align*}
    \frac{\partial \ell_n}{\partial x}(-x,y)    &=y\frac{k_n'(-xy)}{k_n(-xy)}+x\frac{k_n'(x^2)}{k_n(x^2)}.
\end{align*}
Therefore, using \eqref{lxy} and Corollary~\ref{c.kn-der}, we obtain
\begin{align*}
    \bigg|\frac{\left(r_n(-x,y)\frac{\partial \ell_n}{\partial x}(-x,y)\right)^2}{\frac{\partial^2\ell_n}{\partial x\partial y}(-x,-x)}\bigg| 
    &\le \bigg|\frac{2}{\frac{\partial^2\ell_n}{\partial x\partial y}(-x,-x)}\frac{(k_n'(-xy))^2}{k_n(x^2)k_n(y^2)}\bigg|+ \bigg|2r_n^2(-x,y)\frac{\left(\frac{k_n'(x^2)}{k_n(x^2)}\right)^2}{\frac{\partial^2\ell_n}{\partial x\partial y}(-x,-x)}\bigg| \\
    &=\frac{(1-x^2)^2}{(1-xy)^2}O(\tau_n^2)r_n^2(x,y) +O(r_n^2(-x,y))\\
    &=O(\tau_n^2)|r_n(x,y)|^2 = O(\tau_n^2)\alpha^{2\tau+1}.
\end{align*}
Similarly,
    \[
    \frac{(r_n(-x,y)\frac{\partial \ell_n}{\partial y}(-x,y))^2}{(1-r_n^2(-x,y))\frac{\partial^2\ell_n}{\partial x\partial y}(y,y)}=O(\tau_n^2)\alpha^{2\tau+1}.
    \]
Substituting these estimates into \eqref{sigma} yields 
\begin{align*}
    \sigma_n(-x,y)&=\pi^2 \rho^{(1)}_n(-x)\rho^{(1)}_n(y)\left(1+O(\tau_n^2)\alpha^{2\tau+1}\right), \quad (x,y)\in S_n\times S_n.
\end{align*}
Similarly, it holds uniformly for $(x,y)\in S_n\times S_n$ that 
\begin{equation} \label{e.delta-mix}
    \delta_n(-x,y) = O(\tau_n)|r_n(x,y)|= O(\tau_n)\alpha^{\tau+1/2}.
\end{equation}
While the proof is fairly similar, we include the details because there is an artificial singular term that appears when we use $\ell_n$ (instead of $r_n$) to compute $\delta_n$ via \eqref{delta}. To start, by explicit computation (see also \eqref{e.lxy-expand}), we have
    \[
    \frac{\partial^2 \ell_n}{\partial x\partial y}(-x,y) =\frac{k_n'(-xy)}{k_n(-xy)}-xy\frac{k''_n(-xy)}{k_n(-xy)}+ xy\left(\frac{k'_n(-xy)}{k_n(-xy)}\right)^2,
    \]
therefore
\begin{align*}
\begin{split}
    r_n(-x,y)\frac{\partial^2 \ell_n}{\partial x\partial y}(-x,y) 
    &=  \frac {O\big( |k_n'(-xy)|+|k''_n(-xy)|\big)}{\sqrt{k_n(x^2)k_n(y^2)}} + r_n(-x,y)xy\left(\frac{k'_n(-xy)}{k_n(-xy)}\right)^2\\
    &= O(\tau_n) \frac{r_n(x,y)}{(1-xy)^2} + r_n(-x,y)xy\left(\frac{k'_n(-xy)}{k_n(-xy)}\right)^2\\
    &= \frac{O(\tau_nr_n(x,y))}{(1-x^2)(1-y^2)} + r_n(-x,y)xy\left(\frac{k'_n(-xy)}{k_n(-xy)}\right)^2.
\end{split}
\end{align*}
Similarly,
\begin{align*}
    &r_n(-x,y)\frac{\partial\ell_n}{\partial x}(-x,y)\frac{\partial\ell_n}{\partial y}(-x,y)\\
    &= r_n(-x,y)\left(y\frac{k_n'(-xy)}{k_n(-xy)}+x\frac{k_n'(x^2)}{k_n(x^2)}\right)\left(-x\frac{k_n'(-xy)}{k_n(-xy)}-y\frac{k_n'(y^2)}{k_n(y^2)}\right)\\
    &=-r_n(-x,y)xy\left(\frac{k'_n(-xy)}{k_n(-xy)}\right)^2\\
    &\quad  +O(\tau_nr_n(x,y))\bigg(\frac1{1-xy}\frac 1{1-x^2}+\frac1{1-xy}\frac 1{1-y^2}+ \frac 1{1-x^2}\frac 1{1-y^2}\bigg)\\
    &=-r_n(-x,y)xy\left(\frac{k'_n(-xy)}{k_n(-xy)}\right)^2 + \frac{O(\tau_nr_n(x,y))}{(1-x^2)(1-y^2)}.
\end{align*}
Thus, from \eqref{delta} and the above estimates, we have
\begin{align*}
    \delta_n(-x,y)& = \frac{r_n(-x,y)}{\sigma_n(-x,y)} \bigg(\frac{\partial^2\ell_n}{\partial x \partial y}(-x,y)+\frac{\frac{\partial\ell_n}{\partial x}(-x,y)\frac{\partial\ell_n}{\partial y}(-x,y)}{1-r_n^2(-x,y)}\bigg)\\
    &= \frac{1}{\sigma_n(-x,y)}\frac{O(\tau_nr_n(x,y))}{(1-x^2)(1-y^2)}+ \frac{r_n(-x,y)}{\sigma_n(-x,y)}xy\bigg(\frac{k'_n(-xy)}{k_n(-xy)}\bigg)^2\bigg(1-\frac 1{1-r_n^2(-x,y)}\bigg)\\
    &=O(\tau_n r_n(x,y))+ O\bigg(\frac{|r_n(-x,y)^3|}{|\sigma_n(-x,y)|}\left|\frac{k'_n(-xy)}{k_n(-xy)}\right|^2\bigg)\\
    &=O(\tau_n r_n(x,y))+ O(\tau_n^3 r^3_n(x,y))\\
    &=O(\tau_n r_n(x,y)).
\end{align*}
This completes the proof of \eqref{e.delta-mix}.

Now, note that $\Lambda(\delta)=1+O(\delta^2)$ for $\delta$ near $0$, so it follows that
    \[
    \Lambda(\delta_n(-x,y))=1+O(\tau_n^2)\alpha^{2\tau+1}.
    \]
But then 
\begin{align*}
    \rho^{(2)}_n(-x,y)=\frac{1}{\pi^2}\Lambda(\delta_n(-x,y))\frac{\sigma_n(-x,y)}{\sqrt{1-r_n^2(-x,y)}}
    =\rho^{(1)}_n(-x)\rho^{(1)}_n(y)\left(1+O(\tau_n^2)\alpha^{2\tau+1}\right),
\end{align*}
which yields \eqref{rho2-in-mix} when combined with \eqref{rho1-in}.
\end{proof}

\subsection{Correlation functions of real roots of the reciprocal polynomials}
The above analysis applies directly only to estimating the variance of the number of real roots within subsets of $(-1,1)$. To handle the intervals $(-\infty, -1) \cup (1, \infty)$, we consider the reciprocal polynomial
\[Q_n(x):=\frac{x^n}{c_n}\widetilde{P}_n(1/x),\] 
which converts the roots of $\widetilde{P}_n$ in $(-\infty,-1)\cup (1,\infty)$ to the roots of $Q_n$ in $(-1,1)$. Note that
    \[
    Q_n(x)=\sum_{j=0}^n\widetilde{\xi}_{n-j}\frac{c_{n-j}}{c_n} x^j
    \]
is also a Gaussian polynomial. Let $k_{Q_n}(x)$ denote the corresponding variance function; then
    \[
    k_{Q_n}(x)=\sum_{j=0}^n \frac{c_{n-j}^2}{c_n^2}x^j=\frac{x^nk_n(1/x)}{c_n^2}.
    \]
Recall that $I_n=[1-a_n,1-b_n]$, where $a_n=d_n^{-1}=\exp(-\log^{d/4}n)$ and $b_n=d_n/n$. For $x\in I_n$, we will show that $k_{Q_n}(x)$ converges to $\frac{1}{1-x}$ as $n\to\infty$, suggesting that $Q_n$ asymptotically resembles a classical Kac polynomial (this heuristics is well known; see, e.g., \cite{DNV}).

\begin{lemma}\label{l.cor-asymp-out} Let $\rho_{Q_n}^{(1)}$ and $\rho_{Q_n}^{(2)}$ denote the one-point and two-point correlation functions of the real roots of $Q_n$, respectively. Define
\begin{equation} \label{e.en}
    e_n :=\max_{0\le j\le n\sqrt{a_n}}\left |\frac{|c_{n-j}|}{|c_n|}-1 \right|+ \frac 1{\log\log n}.
\end{equation}
Fix $S_n\in \{-I_n, I_n\}$. Then, uniformly for $x\in S_n$, 
    \[
    \rho_{Q_n}^{(1)}(x)=\frac{1}{\pi}\frac{1}{1-x^2}(1+O(e_n)).
    \]
Uniformly for $(x,y)\in S_n\times S_n$,
    \[
    \rho^{(2)}_{Q_n}(x,y)=\frac{1}{\pi^2}\frac{1+f_0(\alpha)}{(1-x^2)(1-y^2)}\left(1+O(\sqrt[16]{e_n})\right),
    \]
where $f_0$ is given by 
\begin{align*}
    f_0(u)=(\sqrt{1-u}+\sqrt{u}\arcsin \sqrt u)\sqrt{1-u}-1,
\end{align*}
as in \eqref{ftau} with $\tau=0$.
Furthermore, there is a positive constant $\alpha_1>0$ (independent of $n$, though possibly depending on the constants and convergence rates in conditions \ref{A1} and \ref{A2}) such that for $\alpha \le \alpha_1$, 
    \[
    \rho_{Q_n}^{(2)}(x,y)-\rho_{Q_n}^{(1)}(x)\rho_{Q_n}^{(1)}(y) = \frac{O(\alpha)}{(1-x^2)(1-y^2)}.
    \]
Finally, uniformly for $(x,y)\in S_n\times S_n$,
    \[
    \rho^{(2)}_{Q_n}(-x,y)-\rho^{(1)}_{Q_n}(-x)\rho^{(1)}_{Q_n}(y)=\frac{1}{\pi^2}\frac{\alpha}{(1-x^2)(1-y^2)}o(1).
    \]
\end{lemma}

\begin{proof}
Write 
    \[
    k_{Q_n}(x)=\sum_{0\le j\le n\sqrt{a_n}} \frac{c_{n-j}^2}{c_n^2}x^j+\sum_{n\sqrt{a_n}<j\le n} \frac{c_{n-j}^2}{c_n^2}x^j.
    \]
By condition \ref{A2}, 
    \[
    \frac{c_{n-j}^2}{c_n^2}=\begin{cases}
    1+o_{n}(1)&\text{if}\quad 0\le j\le n\sqrt{a_n},\\
    O\left(n^{|2\tau|}\right)&\text{if}\quad n\sqrt{a_n}< j\le n,
    \end{cases}
    \]
where $o_{n}(1)\to 0$ as $n\to \infty$. For $x\in I_n^2$, it holds that 
    \[
    n^{|2\tau|+1}x^{n\sqrt{a_n}}\le n^{|2\tau|+1}(1-b_n)^{2n\sqrt{a_n}}=O(e^{-\sqrt{d_n}})=O(a_n).
    \]
Consequently, by putting
    \[
    e_n^0:=\max_{0\le j\le n\sqrt{a_n}}\bigg |\frac{|c_{n-j}|}{|c_n|}-1 \bigg|+ a_n,
    \] 
we see that uniformly for $x\in I_n^2$,
    \[
    k_{Q_n}(x)=\frac{1}{1-x}(1+O(e^0_n))
    \]
and 
    \[
    k_{Q_n}(-x)=O(1)+O(e^0_n)k_{Q_n}(x) = O(e^0_n)k_{Q_n}(x).
    \]
    
For the derivatives of $k_{Q_n}$, we similarly compare them with those of the corresponding variance function for the classical Kac polynomials, leading to analogues of Lemma~\ref{l.kn} (more precisely, Corollary~\ref{c.kn-der}) for $k_{Q_n}$. Note that to account for the derivatives of $k_{Q_n}$, one has to add an $O(1/j)$ term to $e^0_n$ (where $j\ge \log\log n$ as in the proof of Lemma~\ref{l.kn}). Thus, with $e_n$ defined by \eqref{e.en} (which dominates $e^0_n$), it holds uniformly for $x\in I_n^2$ that
\begin{equation} \label{dkq}
    k_{Q_n}^{(j)}(x) = \frac{j!}{(1-x)^{1+j}}(1+O(e_n))\quad \text{and}\quad k_{Q_n}^{(j)}(-x)=O(e_n)k_{Q_n}^{(j)}(x)
\end{equation}
for $j=0,1,...,4$. In other words, $e_n$ plays the same role as $\tau_n$ in the prior treatment of $\widetilde P_n$ within $(-1,1)$.

Let $r_{Q_n}(x,y)$ denote the normalized correlator of $Q_n$; that is,
    \[
    r_{Q_n}(x,y)=\frac{\mathbb E[Q_n(x)Q_n(y)]}{\sqrt{\Var[Q_n(x)]\Var[Q_n(y)]}}=\frac{k_{Q_n}(xy)}{\sqrt{k_{Q_n}(x^2)k_{Q_n}(y^2)}}.
    \]
It follows from \eqref{dkq} that, uniformly for $(x,y)\in S_n\times S_n$, 
    \[
    r_{Q_n}(x,y)=\left[\frac{(1-x^2)(1-y^2)}{(1-xy)^2}\right]^{1/2}(1+O(e_n))=\alpha^{1/2}(x,y)(1+O(e_n)),
    \]
and 
    \[
    r_{Q_n}(-x,y)=O(e_n)r_{Q_n}(x,y).
    \]
The rest of the proof is entirely similar to the previous treatment for $\widetilde{P}_n$, with $\tau=0$ and $C_1=1$.
\end{proof}

\subsection{Asymptotic independence of real roots across distinct core regions}
Our next task is to establish the correlation between the roots of $\widetilde{P}_n$ inside and outside the interval $(-1,1)$.

\begin{lemma} \label{l.cor-in-out}
Let $S_n\in \{-I_n,I_n\}$  and $T_n\in \{-I_n^{-1}, I_n^{-1}\}$. Uniformly for $(x,y)\in S_n\times T_n$, 
\begin{equation} \label{r-mix}
    r_n(x,y)=o(e^{-d_n/2})
\end{equation}
and
\begin{equation} \label{rho2-inout-mix}
    \rho^{(2)}_n(x,y)-\rho^{(1)}_n(x)\rho^{(1)}_n(y)=\rho^{(1)}_n(x)\rho^{(1)}_n(y)o(e^{-d_n/2}),
\end{equation}
where $\rho_n^{(1)}(x)$ satisfies \eqref{rho1-in} and 
\begin{equation} \label{rho1-out}
    \rho_n^{(1)}(y)=\frac{1}{\pi}\frac{1}{y^2-1}(1+O(e_n)),\quad y\in T_n.
\end{equation}
\end{lemma}
\begin{proof} It is well known that $\rho^{(1)}_n(y)=\frac 1 {y^2} \rho^{(1)}_{Q_n}(\frac 1 y)$, which can be demonstrated through a change of variables (see, e.g., \cite{DNV}) or explicit computations.

Recall the definition of $e_n$ from \eqref{e.en}. Applying \eqref{dkq} and proceeding as in the proof of Lemma~\ref{lem.l}, we see that 
    \[
    \rho_n^{(1)}(y)=\frac{1}{\pi}\frac{1}{y^2-1}(1+O(e_n)),\quad y\in T_n,
    \]
which gives \eqref{rho1-out}. It also follows that 
    \[
    \frac{\partial^2\ell_n}{\partial x\partial y}(y,y)=\frac 1{\pi^2(y^2-1)^2}(1+O(e_n)).
    \]

Since $1\ll k_n(x^2)$ for $x\in S_n$, to prove \eqref{r-mix}, it suffices to show that for any bounded integer $m\ge 0$,
\begin{equation} \label{e.k-der-mix}
    |k^{(m)}_n(xy)|= o\left(e^{-2d_n/3}\sqrt{k_n(y^2)}\right).
\end{equation}
To see this, first note that from the polynomial growth of $c_j$, we obtain
    \[
    |k^{(m)}_n(xy)| =O(n^{2\tau+1+m}(|xy|^{n+1}+1)).
    \]
Using $k_n(y)=c_n^2y^{n}k_{Q_n}(1/y)$ and the asymptotic behavior of $k_{Q_n}$ given in \eqref{dkq}, we see that 
    \[
    k_n(y^2) =  c_n^2 \frac{y^{2n+2}}{y^2-1}(1+o(1)) \gg  n^{2\tau}|y|^{2(n+1)}, \quad y\in T_n.
    \]
For $(x,y)\in S_n\times T_n$ and any bounded constant $c$, we have 
    \[
    n^{c}|x|^{n}=O(n^{c}(1-b_n)^n)=o(e^{-2d_n/3}), \quad \text{and similarly}\quad \frac{n^{c}}{|y|^{n}} = o(e^{-2d_n/3}).
    \]
Consequently,
\begin{align*}
    \bigg|\frac{k^{(m)}_n(xy)}{\sqrt{k_n(y^2)}}\bigg|
    &=O\bigg(\frac{n^{O(1)}}{\sqrt{k_n(y^2)}} + \frac{n^{O(1)}|x|^{n+1}|y|^{n+1}}{\sqrt{k_n(y^2)}}\bigg)\\
    &\le O\bigg(\frac{n^{O(1)}}{|y|^{n+1}}\bigg) +O\left(n^{O(1)}|x|^{n+1}\right) \\
    &= o(e^{-2d_n/3}),
\end{align*}
completing the proof of \eqref{e.k-der-mix}, and  thus verifying \eqref{r-mix}.  

To prove \eqref{rho2-inout-mix}, we first employ the explicit computation of $\frac{\partial \ell_n}{\partial x}$ (see \eqref{e.lx}) to estimate
    \[
    \bigg|\frac{\left(r_n(x,y)\frac{\partial \ell_n}{\partial x}(x,y)\right)^2}{\frac{\partial^2\ell_n}{\partial x\partial y}(x,x)} \bigg| \le \bigg|\frac{2y^2}{\frac{\partial^2\ell_n}{\partial x\partial y}(x,x)}\frac{(k_n'(xy))^2}{k_n(x^2)k_n(y^2)} \bigg| + \bigg|2r_n^2(x,y)\frac{\left(\frac{k_n'(x^2)}{k_n(x^2)}\right)^2}{\frac{\partial^2\ell_n}{\partial x \partial y}(x,x)}\bigg|.
    \]
Using \eqref{e.k-der-mix}, we have
    \[
    \frac{|k_n'(xy)|^2}{k_n(x^2)k_n(y^2)}=o(e^{-d_n}),\quad (x,y)\in S_n\times T_n.
    \]
It follows from Lemma~\ref{lem.l}  and  Corollary~\ref{c.kn-der} that
    \[
    \frac{2y^2}{\frac{\partial^2\ell_n}{\partial x\partial y}(x,x)}=o(1) \quad \text{and}\quad \frac{\left(\frac{k_n'(x^2)}{k_n(x^2)}\right)^2}{\frac{\partial^2\ell_n}{\partial x\partial y}(x,x)}=O(1).
    \]
Together with \eqref{r-mix}, we obtain
\begin{equation} \label{rm-x}  
    1-\frac{(r_n(x,y)\frac{\partial \ell_n}{\partial x}(x,y))^2}{(1-r_n^2(x,y))\frac{\partial^2\ell_n}{\partial x\partial y}(x,x)}=1+o(e^{-d_n}),\quad (x,y)\in S_n\times T_n.
\end{equation}
Similarly, we also have
\begin{equation} \label{rm-y} 
    1-\frac{(r_n(x,y)\frac{\partial \ell_n}{\partial y}(x,y))^2}{(1-r_n^2(x,y))\frac{\partial^2\ell_n}{\partial x\partial y}(y,y)}=1+o(e^{-d_n}),\quad (x,y)\in S_n\times T_n.
\end{equation}
Substituting \eqref{rm-x} and \eqref{rm-y} into \eqref{sigma} yields
    \[
    \sigma_n(x,y)=\pi^2 \rho_n^{(1)}(x)\rho_n^{(1)}(y)(1+o(e^{-d_n})), \quad (x,y)\in S_n\times T_n.
    \]

In the same manner, we can show that 
    \[
    \delta_n(x,y)=o(e^{-d_n/2}), \quad (x,y)\in S_n\times T_n.
    \]
To prove this, we proceed similarly to the proof of \eqref{e.delta-mix} in Lemma~\ref{l.cor-asymp-mix}. Recall from \eqref{e.lxy-expand} that
    \[
    \frac{\partial^2 \ell_n}{\partial x\partial y}(x,y) =\frac{k_n'(xy)}{k_n(xy)}+xy\frac{k''_n(xy)}{k_n(xy)}- xy\left(\frac{k'_n(xy)}{k_n(xy)}\right)^2,
    \]
therefore,
\begin{align*}
\begin{split}
    r_n(x,y)\frac{\partial^2 \ell_n}{\partial x\partial y}(x,y) 
    &=  \frac {O\big( |k_n'(xy)|+|k''_n(xy)|)  \big)}{\sqrt{k_n(x^2)k_n(y^2)}} - r_n(x,y)xy\left(\frac{k'_n(xy)}{k_n(xy)}\right)^2\\
    &= o(e^{-d_n/2}) - r_n(x,y)xy\left(\frac{k'_n(xy)}{k_n(xy)}\right)^2.
\end{split}
\end{align*}
On the other hand, using \eqref{e.lx} gives
    \[
    r_n(x,y)\frac{\partial\ell_n}{\partial x}(x,y)\frac{\partial\ell_n}{\partial y}(x,y)= r_n(x,y)\bigg(y\frac{k_n'(xy)}{k_n(xy)}-x\frac{k_n'(x^2)}{k_n(x^2)}\bigg)\bigg(x\frac{k_n'(xy)}{k_n(xy)}-y\frac{k_n'(y^2)}{k_n(y^2)}\bigg).
    \]
As we will see, the main term on the right-hand side is $r_n(x,y)xy\big(\frac{k'_n(xy)}{k_n(xy)}\big)^2$. In the estimate below, we will use the crude estimate $|k'_n(t^2)/k_n(t^2)|\le n$ for all $t\in \mathbb R$. Combining with \eqref{e.k-der-mix}, it follows that
\begin{align*}
    r_n(x,y)\frac{\partial\ell_n}{\partial x}(x,y)\frac{\partial\ell_n}{\partial y}(x,y) &=r_n(x,y)xy\left(\frac{k'_n(xy)}{k_n(xy)}\right)^2+ o(e^{-d_n/2}).
\end{align*}
Since $r_n(x,y) = o(e^{-d_n/2})$ and $\sigma_n(x,y)\gg 1$ (as proved above), using \eqref{delta} we obtain
\begin{align*}
    \delta_n(x,y)&=\frac{1}{\sigma_n(x,y)}\bigg[r_n(x,y)xy\bigg(\frac{k'_n(xy)}{k_n(xy)}\bigg)^2\bigg(\frac 1{1-r_n(x,y)^2}-1\bigg)+o(e^{-d_n/2})\bigg]\\
    &=\frac{1}{\sigma_n(x,y)}\bigg[ \left(\frac{k'_n(xy)}{k_n(xy)}\right)^2O(|r_n(x,y)|^3)+o(e^{-d_n/2})\bigg]\\
    &=o(e^{-d_n/2}),
\end{align*}
which proves the desired claim $\delta_n(x,y)=o(e^{-d_n/2})$. Together with \eqref{rho2}, we deduce \eqref{rho2-inout-mix} as desired.
\end{proof}

\section{Proof of Theorem \ref{t.maingauss}} \label{s.mainthmproof}
From Lemma \ref{l.cor-in-out}, it follows that the numbers of real roots of $\widetilde{P}_n$ in $S_n\in \{-I_n,I_n\}$ and in $T_n\in \{-I_n^{-1},I_n^{-1}\}$ are asymptotically independent. Indeed, on account of \eqref{rho2-inout-mix} and \eqref{rho1-in}, we have 
\begin{align*}
    \Cov[\widetilde{N}_n (S_n),\widetilde{N}_n (T_n)]&=\int_{S_n}dx\int_{T_n}\left(\rho^{(2)}_n(x,y)-\rho^{(1)}_n(x)\rho^{(1)}_n(y)\right)dy\\
    &=o(e^{-d_n/2})\int_{S_n}\frac{1}{1-x^2}dx\int_{T_n}\frac{1}{y^2-1}dy\\
    &=o(e^{-d_n/2})O(\log^2n)=o(1).
\end{align*}
This gives
    \[
    \Var[\widetilde{N}_n(\mathcal I_n)]=\Var[\widetilde{N}_n(-I_n\cup I_n)]+\Var[\widetilde{N}_n(-I_n^{-1}\cup I_n^{-1})]+o(1).
    \]
Therefore, the proof of Theorem \ref{t.maingauss} naturally splits into two parts, as addressed in Lemma~\ref{l.var-in} and Lemma~\ref{l.var-out}. Before stating and proving these lemmas, we first collect some fundamental properties of $f_\tau$ (as defined in \eqref{ftau}), which will be useful throughout the proof.

\begin{lemma}\label{l.f-fact}
For $\tau>-1/2$, it holds that $\sup_{u\in [0,1]}|f_\tau(u)|<\infty$ and
\begin{align*}
    f_{\tau}(u)=\begin{cases} 
    O(u^{2\tau+1})&\text{as}\quad u\to 0^+,\\
    -1+O(\sqrt{1-u})&\text{as}\quad u\to 1^-.
    \end{cases}
\end{align*}
Furthermore, given any $\varepsilon'\in (0,1/2)$, $f_\tau$ is real analytic on $(\varepsilon',1-\varepsilon')$. In particular, the equation $f_\tau(u)=0$ has at most finitely many real roots in $(0,1)$, each of which has a finite vanishing order. 
\end{lemma}
\begin{proof}
The estimates near $0$ and $1$ for $f_\tau$ follow directly from \eqref{ftau} and Taylor expansions. Specifically, for $u$ near $0$, one has $f_\tau(u)=2\tau^2 u^{2\tau+1}(1+o(1))$ if $\tau\ne0$, and for $\tau=0$, $f_0(u)=(-\frac{1}{3}) u^2(1+o(1))$. These endpoint estimates demonstrate that the roots do not accumulate at $0$ or $1$ (see Figure \ref{fig5}).

\begin{figure}[htbp]
\centering
\includegraphics[scale=1]{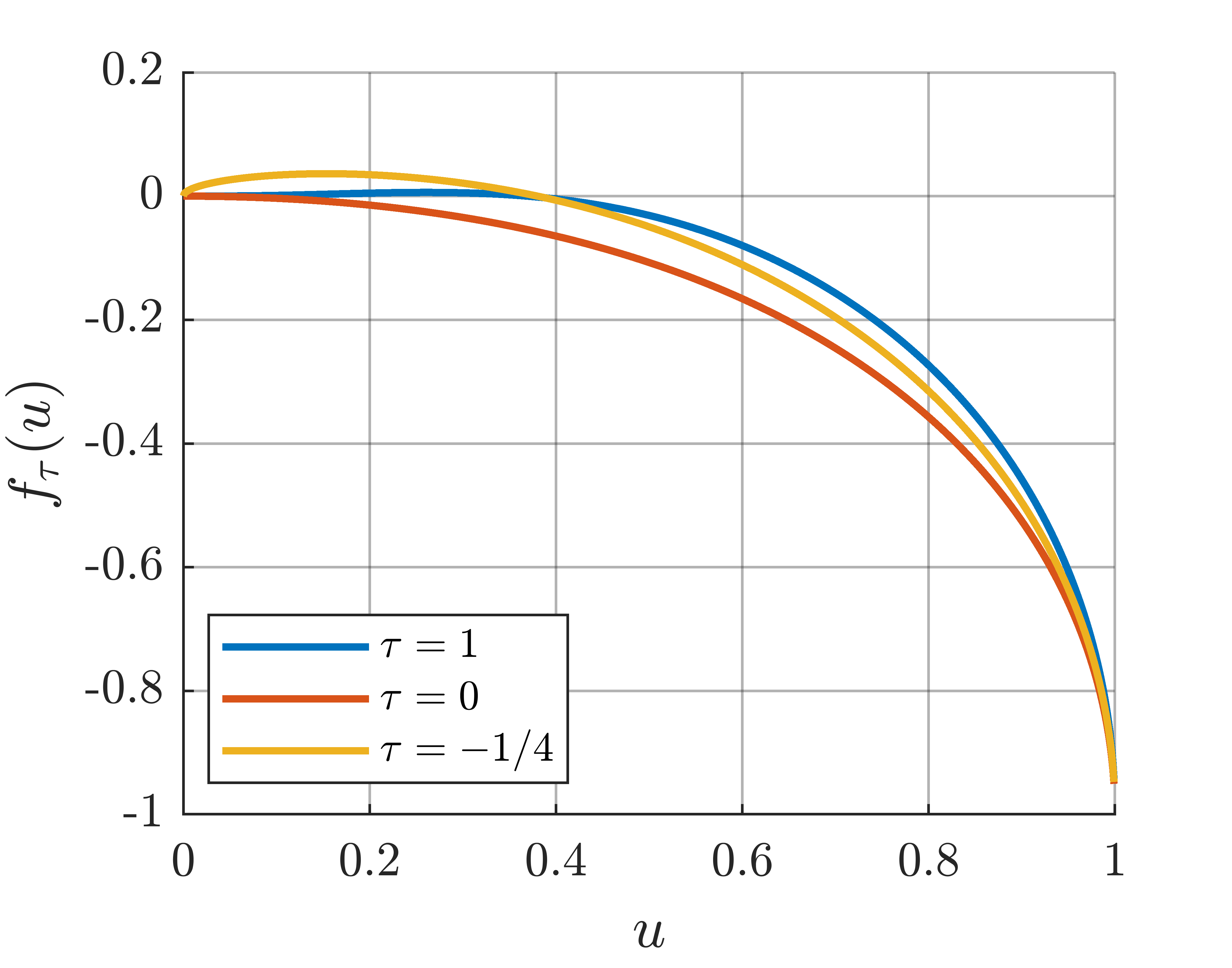}
\caption{Plots of $f_\tau$ on $[0,1]$ when $\tau= 1, 0, -1/4$.}
\label{fig5}
\end{figure}

Recall the definition of $\Delta_\tau(u)$ in \eqref{del}. Note that $0\ge \Delta_\tau(u)\ge -1$, and these inequalities are strict for $u\in (0,1)$ (see Figure~\ref{fig6}).

\begin{figure}[htbp]
\centering
\includegraphics[scale=1]{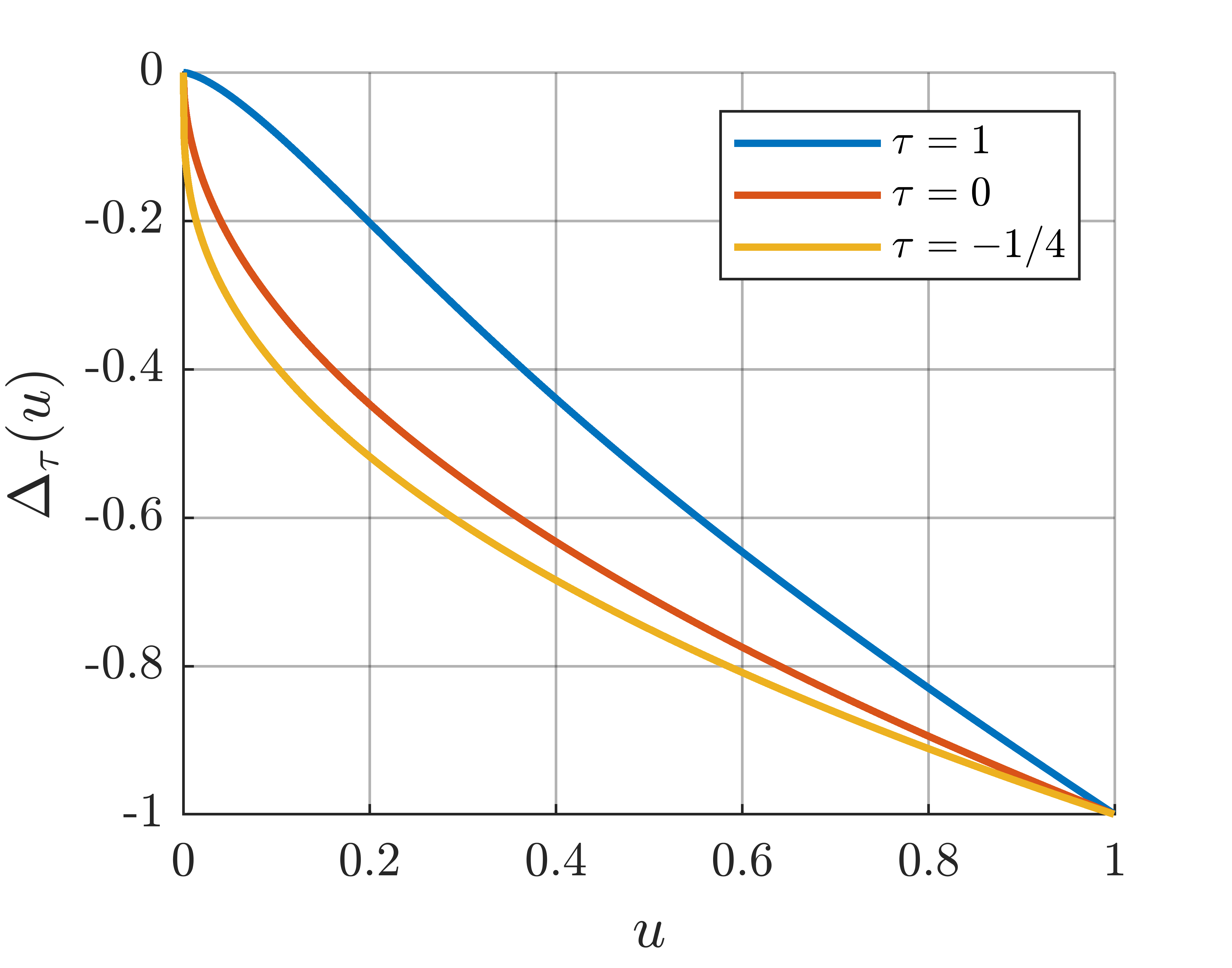}
\caption{Plots of $\Delta_\tau$ on $[0,1]$ when $\tau= 1, 0, -1/4$.}
\label{fig6}
\end{figure}
 
Writing $\Delta_\tau =u^{\tau+1/2}\frac{\rm Num}{\rm Denom}$, we show that ${\rm Num}<0$ while $u^{\tau+\frac 1 2}{\rm Num}+{\rm Denom}>0$ for $u\in (0,1)$. Through examination, we first see that ${\rm Num}$ is strictly increasing on $(0,1)$, so ${\rm Num}<{\rm Num}(1)=0$. Now,
    \[
    u^{\tau+\frac 1 2}\text{Num} + \text{Denom} = (1-u^{2\tau+1})(1+u^{\tau+\frac 32})- (2\tau+1)(1-u)(u^{\tau+\frac 1 2}+u^{2\tau+1})=:h(u).
    \]
If $\tau \ge 0$, it can be verified that $(1-u^{2\tau+1}) - (2\tau+1)(1-u)u^\tau$ is decreasing for $u\in (0,1)$. Hence, $(1-u^{2\tau+1}) \ge (2\tau+1)(1-u)u^\tau$. Consequently, 
\begin{align*}
    h(u) &\ge (2\tau+1)(1-u) (u^\tau+u^{2\tau+\frac 32} -u^{\tau+\frac 1 2}-u^{2\tau+1})\\
    &=(2\tau+1)u^\tau(1-u)(1-\sqrt{u})(1-u^{\tau+1})\\
    &>0.
\end{align*}
If $-\frac 1 2 <\tau<0$, we will show that $h$ is strictly decreasing on $(0,1)$. One has 
\begin{align*}
    h'(u)&= (2\tau+2)(\tau+\frac 3 2)u^{\tau+\frac 12}+(2\tau+2)(2\tau+1)u^{2\tau+1}-(3\tau+\frac  52)u^{3\tau+\frac 32}  \\
    &\quad -(2\tau+2)(2\tau+1)u^{2\tau}-(2\tau+1)(\tau+\frac 12)u^{\tau-\frac 12},
\end{align*}
and 
\begin{align*}
    \frac{d}{du}(u^{-2\tau}h'(u)) &=(2\tau+2)(\tau+\frac 3 2)(\frac 12-\tau)u^{-\tau-\frac 12}+(2\tau+2)(2\tau+1)\\
    &\quad -(3\tau+\frac  52)(\tau+\frac 32)u^{\tau+\frac 12}  + (2\tau+1)(\tau+\frac 12)(\tau+\frac 12)u^{-\tau-\frac 32}\\
    &> u^{\tau+\frac 12}\bigg[(2\tau+2)(\tau+\frac 3 2)(\frac 12-\tau)  +(2\tau+2)(2\tau+1)   \\
    &\quad - (3\tau+\frac  52)(\tau+\frac 32)+(2\tau+1)(\tau+\frac 12)(\tau+\frac 12)\bigg]\\
    &=0.
\end{align*}
Thus, $u^{-2\tau}h'(u)$ is strictly increasing on $(0,1)$, and so $u^{-2\tau}h'(u)<h'(1)=0$ for $u\in (0,1)$. Consequently, $h'<0$ and $h$ is strictly decreasing on $(0,1)$, which implies $h(u)>h(1)=0$. This completes the proof of the claimed estimates for $\Delta_\tau$.

By continuity, the above considerations imply $\max_{u\in [\varepsilon',1-\varepsilon']}|\Delta_\tau(u)|<1$. Consequently, using the principal branch of $\log$, it is clear from the definition that $f_\tau$ has analytic continuation to a neighborhood of $[\varepsilon',1-\varepsilon']$ in $\mathbb C$, thus the claimed properties regarding the real zeros of $f_\tau$ follow.
\end{proof}

\begin{lemma}\label{l.var-in}
It holds that 
\begin{equation} \label{var-in}
    \Var[\widetilde{N}_n(S_n)] = \left(\kappa_\tau+o(1)\right)\log n, \quad S_n\in \{-I_n,I_n\},
\end{equation}
and 
\begin{equation} \label{var-in-cup}
    \Var[\widetilde{N}_n(-I_n\cup I_n)] = \left(2\kappa_\tau+o(1)\right)\log n.
\end{equation}
\end{lemma}
\begin{proof} 
We start with \eqref{var-in}. Let $\varepsilon'>0$ be arbitrary. It suffices to show that for $n$ large enough (depending on $\varepsilon'$), the following holds
    \[
    |\Var[\widetilde{N}_n(S_n)] - \kappa_\tau \log n| = O(\varepsilon' \log n).
    \]

By \eqref{mean} and \eqref{rho1-in}, 
\begin{align*}
    \mathbb E[\widetilde{N}_n (S_n)]
    =\int_{S_n}\frac{\sqrt{2\tau+1}}{\pi}\frac{1}{1-x^2}\left(1+o(1)\right) dx=\left(\frac{\sqrt{2\tau+1}}{2\pi}+o(1)\right)\log n.
\end{align*}
Now, using the change of variables $x=\tanh t$ and $y=\tanh s$, we see that 
    \[
    \frac{dxdy}{(1-x^2)(1-y^2)}=dtds \quad \text{and}\quad \alpha=\sech^2(s-t), \quad  (t,s)\in J_n\times J_n,
    \]
where 
    \[
    J_n:=\begin{cases}
    (\frac 12 \log \frac{2-a_n}{a_n},\frac 12\log\frac{2-b_n}{b_n}) &\text{if}\quad S_n=I_n,\\
    (\frac 12 \log \frac{b_n}{2-b_n},\frac 12\log\frac{a_n}{2-a_n})&\text{if}\quad S_n=-I_n,
    \end{cases}
    \]
and it is clear that
    \[
    |J_n|=\frac{1}{2}\log n-\log^{\frac d4}n+o(1).
    \]

Recall the constant $\alpha_0>0$ from Lemma~\ref{l.cor-asymp}. Then there is a constant $M_0>0$ such that $\sech(t)< \alpha_0$ is equivalent to $|t|> M_0$. It follows that
\begin{align*}
    &\iint_{\{S_n\times S_n: \alpha< \alpha_0\}}|\rho_n^{(2)}(x,y)-\rho_n^{(1)}(x)\rho_n^{(1)}(y)|dxdy\\
    &= O\bigg(\iint_{\{J_n\times J_n: |s-t| > M_0\}}\sech^{4\tau+2}(s-t)dsdt\bigg)\\
    &=  O\bigg(\int_{M_0}^{|J_n|}(|J_n|-v)\sech^{4\tau+2}(v)dv\bigg)\\
    &= |J_n| O\left(\int_{M_0}^{\infty}\sech^{4\tau+2}(v)dv\right) + O(1).
\end{align*}
 We can now refine $\alpha_0$ (making it smaller) so that
    \[
    \int_{M_0}^\infty\sech^{4\tau+2}(v)dv  <\varepsilon',
    \]
and  it follows that
\begin{equation}\label{e.var-in-smallapha}
    \iint_{\{S_n\times S_n: \alpha<\alpha_0\}}|\rho_n^{(2)}(x,y)-\rho_n^{(1)}(x)\rho_n^{(1)}(y)|dxdy = O(\varepsilon' \log n). 
\end{equation}

We now estimate the integral
    \[
    \iint_{\{S_n\times S_n: \alpha\ge \alpha_0\}}|\rho_n^{(2)}(x,y)-\rho_n^{(1)}(x)\rho_n^{(1)}(y)|dxdy.
    \]
We partition the integration region into two subsets, $(I)$ and $(II)$, defined as follows:
\begin{align*}
    (I)&:=\{(x,y)\in S_n\times S_n: \alpha\ge \alpha_0, |f_{\tau}(\alpha)| < \sqrt[32]{\tau_n}\}, \\
    (II)&:=\{(x,y)\in S_n\times S_n: \alpha\ge \alpha_0, |f_{\tau}(\alpha)| \ge \sqrt[32]{\tau_n}\}.
\end{align*}
We will use the same change of variable $x=\tanh t$ and $y=\tanh s$, so that $\alpha=\sech^2(s-t)$. 

For $(I)$, using Lemma~\ref{l.f-fact}, it is clear that the set $E=\{u\in  [\alpha_0,1]: |f(u)| < \sqrt[32]{\tau_n}\}$ can be covered by a union of finitely many subintervals of $[\alpha_0,1]$, each with a length of $o(1)$. Let $F=\{v: \sech^2(v)\in E\}$. Then it is clear that $F$ may be covered by a union of finitely many intervals, each with a length of $o(1)$. (The implicit constant may depend on $\alpha_0$). Thus, using the boundedness of $f_\tau$ and Lemma~\ref{l.cor-asymp}, we have
\begin{align*}
    \iint_{(I)} |\rho_n^{(2)}(x,y)-\rho_n^{(1)}(x)\rho_n^{(1)}(y)|dxdy  &\le O\bigg(\iint_{(I)} \rho_n^{(1)}(x)\rho_n^{(1)}(y)dxdy\bigg)  \\
    &\le \iint_{\{J_n\times J_n: |s-t|\in F\}}O(1) dsdt \\
    &= o(|J_n|)\\
    &=o(\log n).
\end{align*}

For $(II)$, we note that  
    \[
    (1+f_\tau(\alpha))(1+O(\sqrt[16]{\tau_n}))=1+f_\tau(\alpha)+O(|f_\tau(\alpha)|\sqrt[32]{\tau_n}).
    \]
Thus, using Lemma~\ref{l.cor-asymp}, the above estimates (for region $(I)$), and the boundedness of $f_\tau$, we obtain
\begin{align*}
    &\iint_{(II)} \rho_n^{(2)}(x,y)-\rho_n^{(1)}(x)\rho_n^{(1)}(y)dxdy \\
    &=\iint_{(II)}  \Big(f_\tau(\alpha) +O(\sqrt[32]{\tau_n})|f_\tau(\alpha)|\Big) \rho_n^{(1)}(x)\rho_n^{(1)}(y)dxdy \\
    &=o(\log n)+\iint_{\{S_n\times S_n: \alpha\ge \alpha_0\}}  \Big(f_\tau(\alpha) +O(\sqrt[32]{\tau_n})|f_\tau(\alpha)|\Big) \rho_n^{(1)}(x)\rho_n^{(1)}(y)dxdy.
\end{align*}
Making the change of variables $x=\tanh t$ and $y=\tanh s$ again, from Lemma~\ref{l.f-fact}, we know that $\int_{0}^{\infty}|f_{\tau}(\sech^2v)|dv$ and $\int_{0}^{\infty}v|f_{\tau}(\sech^2v)|dv$ both converge. It follows that
\begin{align*}
    &\iint_{\{S_n\times S_n: \alpha\ge \alpha_0\}}  f_\tau(\alpha)   \rho_n^{(1)}(x)\rho_n^{(1)}(y)dxdy\\
    &=\frac{2\tau+1}{\pi^2} \iint_{\{J_n\times J_n: |t-s|\le M_0\}}  f_\tau(\sech^2(t-s)) dsdt\\
    &=\frac{2(2\tau+1)}{\pi^2}\int_{0}^{M_0}(|J_n|-v)f_{\tau}(\sech^2v)  dv \\
    &=\frac{2\tau+1}{\pi^2}\left(\int_{0}^{\infty}f_{\tau}(\sech^2 v)dv+O(\varepsilon')\right) \log n + o(\log n).
\end{align*}

A similar computation works for $|f_\tau|$, giving a contribution of order 
    \[
    O(\sqrt[32]{\tau_n}\log n)=o(\log n),
    \]
and we obtain
\begin{align*}
    &\iint_{\{S_n\times S_n:\alpha \ge \alpha_0\}}\Big(\rho^{(2)}_n(x,y)-\rho^{(1)}_n(x)\rho^{(1)}_n(y)\Big)dxdy\\
    &=\left(\frac{2\tau+1}{\pi^2}\int_0^\infty f_{\tau}(\sech^2 v)dv +O(\varepsilon')\right) \log n.
\end{align*}

Combining these with \eqref{e.var-in-smallapha} and \eqref{var}, we get 
\begin{align*}
    \Var[\widetilde{N}_n(S_n)]=\left(\frac{2\tau+1}{\pi}\int_{0}^{\infty} f_{\tau}(\sech^2v)dv+\frac{\sqrt{2\tau+1}}{2}+O(\varepsilon')\right)\frac{1}{\pi}\log n,
\end{align*}
for any $\varepsilon'>0$, which gives \eqref{var-in}, recalling the definition of $\kappa_\tau$ in \eqref{kappa}. 

We now prove \eqref{var-in-cup}. Using \eqref{rho2-in-mix}, we get 
\begin{align*}
\Cov[\widetilde{N}_n (-I_n),\widetilde{N}_n (I_n)]
    &=\int_{-I_n}dx\int_{I_n}\left(\rho^{(2)}_n(x,y)-\rho^{(1)}_n(x)\rho^{(1)}_n(y)\right)dy\\
    &=\iint_{I_n\times I_n}\Big(\rho^{(2)}_n(-x,y)-\rho^{(1)}_n(-x)\rho^{(1)}_n(y)\Big)dxdy\\
    &=\iint_{I_n\times I_n}\frac{2\tau+1}{\pi^2}\frac{\alpha^{2\tau+1}(x,y)}{(1-x^2)(1-y^2)} o(1)dxdy\\
    &=o(1)\int_{0}^{|J_n|}(|J_n|-v)(\sech^2v)^{2\tau+1}dv\\
    &=o(1)\log n.
\end{align*}
Combining with \eqref{var-in} and 
    \[
    \Var[\widetilde{N}_n(-I_n\cup I_n)]=\Var[\widetilde{N}_n(-I_n)]+\Var[\widetilde{N}_n(I_n)]+2\Cov[\widetilde{N}_n(-I_n), \widetilde{N}_n(I_n)],
    \]
we deduce \eqref{var-in-cup} as desired.
\end{proof}

\begin{lemma} \label{l.var-out}
It holds that 
    \[
    \Var[\widetilde{N}_n (T_n)]=\left[\frac{1}{\pi}\left(1-\frac{2}{\pi}\right)+o(1)\right]\log n,\quad T_n\in \{-I_n^{-1},I_n^{-1}\},
    \]
and 
    \[
    \Var[\widetilde{N}_n (-I_n^{-1}\cup I_n^{-1})]=\left[\frac{2}{\pi}\left(1-\frac{2}{\pi}\right)+o(1)\right]\log n.
    \]
\end{lemma}
\begin{proof} The proof follows the same argument as that of Lemma~\ref{l.var-in}, specialized to the case $\tau=0$, and relies on Lemma~\ref{l.cor-asymp-out}.
\end{proof}

\begin{appendix}
\section{}  \label{a.example}
This appendix provides an example showing that the asymptotic condition \eqref{e.polyasymp} does not imply the concentration condition \eqref{e.OV}.

Consider the sequence $c_j =j^\tau\big(1+\frac{(-1)^{j}}{\log j}\big)$, which satisfies \eqref{e.polyasymp}, but we will show that it does not satisfy \eqref{e.OV}. 

First, observe that for $n-e^{(\log n)^{1/5}} \ge j\ge n-ne^{-(\log n)^{1/5}}$, we have
    \[
    \frac {j^\tau}{n^\tau}= \bigg(1+o\bigg(\frac 1{\log n}\bigg)\bigg)^\tau = 1+o\left(\frac 1{\log n}\right).
    \]
Now, if $n-j$ is odd, then we claim that
    \[
    \bigg|\frac{1+\frac{(-1)^j}{\log j}}{1+\frac{(-1)^n}{\log n}} -1\bigg|\ge \frac 1{\log n}.
    \]
To see this, consider two cases. If $n$ is even, then 
    \[
    \frac{1-\frac 1{\log j}}{1+\frac 1{\log n}} \le 1-\frac 1{\log j}\le 1-\frac 1{\log n}.
    \]
If $n$ is odd, then
    \[
    \frac{1+\frac 1{\log j}}{1-\frac 1{\log n}} \ge 1+\frac 1{\log j}\ge 1+\frac 1{\log n}.
    \]
Thus, the claim holds. It follows that for large $n$,
    \[
    \bigg|\frac {|c_j|}{|c_n|} -1\bigg| = \bigg|\frac {j^\tau}{n^\tau}\frac{1+\frac{(-1)^j}{\log j}}{1+\frac{(-1)^n}{\log n}}-1\bigg|\ge \frac 1 2 \frac 1{\log n}=\frac 1 2 e^{-\log\log n},
    \]
so condition~\eqref{e.OV} is not satisfied.

We remark that this example can be modified to show that the asymptotic condition \eqref{e.polyasymp} does not imply conditions similar to \eqref{e.OV}, where one requires a decay estimate for $\left|\frac{|c_j|}{|c_n|}-1\right|$ (as $n\to\infty$) that is stronger than the uniform decay rate for $|\frac{j^\tau}{n^\tau}-1|$ (over the range of $j$ under consideration).
\end{appendix}

\subsection*{Acknowledgments}
The authors thank the two anonymous referees for their insightful comments and suggestions, which substantially improved the paper.
\subsection*{Funding statement}
The first author has been supported in part by NSF grant DMS-1800855.


\end{document}